\documentclass{amsart}
\usepackage{amsthm}
\usepackage{amsmath}
\usepackage{amssymb}
\usepackage{enumerate}
\usepackage{enumitem}
\usepackage{mathrsfs}
\usepackage{float}
\usepackage{graphicx}
\usepackage{mathtools}
\usepackage{hyperref}
\usepackage{caption}
\usepackage{subcaption}
\usepackage{faktor}
\usepackage[dvipsnames]{xcolor}
\usepackage{lineno}

\newtheorem{theorem}{Theorem}

\newtheorem{proposition}[theorem]{Proposition}
\newtheorem{lemma}[theorem]{Lemma}
\newtheorem{definition}[theorem]{Definition}
\newtheorem{example}[theorem]{Example}
\newtheorem{remark}[theorem]{Remark}
\newtheorem{question}{Question}

\numberwithin{theorem}{section}
\numberwithin{equation}{section}

\newcounter{relctr} 
\everydisplay\expandafter{\the\everydisplay\setcounter{relctr}{0}} 

\AtBeginDocument{} 

\makeatletter
\newcommand{\mylabel}[2]{#2\def\@currentlabel{#2}\label{#1}}
\makeatother

\allowdisplaybreaks

\begin{document}
	\title[Weighted $k$-plane transform]{On the mapping properties of a radially power-weighted \texorpdfstring{$k$}{k}-plane transform}
	\author[A Deshmukh]{Aniruddha Deshmukh}
	\address{Department of Mathematics, Indian Institute of Technology Indore, Khandwa Road, Simrol, Indore-453552, Madhya Pradesh, India.}
	\email{aniruddha480@gmail.com; phd2001241001@iiti.ac.in}
	\author[A Kumar]{Ashisha Kumar}
	\address{Department of Mathematics, Indian Institute of Technology Indore, Khandwa Road, Simrol, Indore-453552, Madhya Pradesh, India.}
	\email{akumar@iiti.ac.in}
	\date{}
	\begin{abstract}
		In this article, we study a weighted analogue of the $k$-plane transform, where the weight on the $k$-dimensional measure is given by a radial power of distance. Particularly, we first consider the existence of such operators on Lebesgue spaces imbibed with radial power weights. Then, we proceed to see the mapping properties of the said transform when it acts on radial functions. In this regard, we prove certain weighted $L^p$-improving boundedness and end-point Lorentz space estimates. Finally, we look at the weighted $L^p$-$L^p$ boundedness of the operator acting on general functions, and evaluate its operator norm.
	\end{abstract}
	\maketitle
	\section{Introduction}
		\label{IntroductionSection}
		The study of $k$-plane transforms arose from the question of reconstructing functions from their averages over hyperplanes. Formally, the $k$-plane transform of a function $f: \mathbb{R}^n \rightarrow \mathbb{C}$ is a function $R_kf$ on the affine Grassmannian $G \left( n, k \right)$ of affine $k$-dimensional planes in $\mathbb{R}^n$, defined by
		\begin{equation}
			\label{KPlaneTransformDefinition}
			R_kf \left( \xi \right) := \int\limits_{\xi} f \left( x \right) \mathrm{d}_{\xi}x,
		\end{equation}
		where, $\mathrm{d}_{\xi}x$ is the $k$-dimensional Lebesgue measure on the $k$-plane $\xi$.
		
		Over the past century, the $k$-plane transform has found its use in various fields such as medical imaging and mathematical tomography. For a brief survey of these topics, we refer the reader to \cite{SmithSolmon}, \cite{MarkoeBook}, and the references therein. Among these transforms, the special cases $k = 1$ (called the X-Ray transform), and $k = n - 1$ (called the Radon transform) have attained much attention. Not only do these transforms have nice applications but also they are easy to study from a perspective of Analysis.
		
		While the X-Ray and Radon transforms are foundational to computer assisted tomography (CAT), various other physical systems have demanded variants of these transforms. Particularly, it is often useful to consider certain weights with the measure of the hyperplanes (or lines) that contribute to attenuation or strengthening of the ``X-rays". Formally, the attenuated X-Ray transform is defined by
		\begin{equation}
			\label{AttenuatedXRayDefinition}
			R_{\mu}f \left( \omega, s \right) := \int\limits_{\mathbb{R}} f \left( s + t\omega \right) e^{- \int\limits_{t}^{\infty} \mu \left( s + \tau \omega \right) \mathrm{d}\tau} \mathrm{d}t,
		\end{equation}
		where, $\omega \in \mathbb{S}^{n - 1}$, $s \in \omega^{\perp}$, and $\mu$ is the attenuation. A special case of Equation \eqref{AttenuatedXRayDefinition} that is of interest is when the attenuation $\mu$ is constant. This gives rise to what is known as the exponential X-Ray transform. Formally, it is defined as
		\begin{equation}
			\label{ExponentialXRayDefinition}
			R_{\mu}f \left( \omega, s \right) := \int\limits_{\mathbb{R}} f \left( s + t \omega \right) e^{\mu t} \mathrm{d}t,
		\end{equation}
		where, $\mu \in \mathbb{R}$ is a constant. It is to be noticed that when $\mu > 0$, the X-rays are actually amplified as they pass through the medium, instead of being attenuated. For developments on the attenuated and exponential Radon (X-ray) transforms, we refer the reader to \cite{Clough}, \cite{Finch}, \cite{Hertle}, \cite{Holman}, \cite{MarkoeAttenuated}, \cite{NattererAttenuatedRadon}, \cite{NattererInversion}, \cite{Puro}, \cite{Riguad}, \cite{Tretiak}, \cite{YuXRay2020}, and \cite{YuAttenuatedXRay}.
		
		It is clear that the attenuation (or amplification) is obtained by inducing the measure of the X-rays by a weight. In the case of Equations \eqref{AttenuatedXRayDefinition} and \eqref{ExponentialXRayDefinition}, the weights are exponential. These definitions have motivated researchers to consider Radon transforms with more general weights in the measure of the hyperplanes. We refer the reader to \cite{Ambartsoumian}, \cite{Anikonov}, \cite{Boman}, \cite{LiWtXRay}, \cite{NovikovWeRadonInversion}, \cite{NovikovWtRadonDE}, and the references therein for a survey on the topic of weighted Radon transforms.
		
		While some results, especially about the inversion of the weighted Radon transforms, are available in literature, not much focus has been given to these transforms when the weights are powers of distances. Our aim is to study such transforms and get their mapping properties on weighted Lebesgue spaces. Particularly, we shall consider the transform,
		\begin{equation}
			\label{OurTransformDefinition}
			R_{\alpha}f \left( \pi, u \right) := \int\limits_{\pi} f \left( x + u \right) \| x \|^{\alpha - k} \mathrm{d}x,
		\end{equation}
		where, $\pi \in G_{n, k}$, the (compact) Grassmannian of $k$-dimensional subspaces in $\mathbb{R}^n$, $u \in \pi^{\perp}$, and $\alpha > 0$. From a functional analytic perspective, operators given in Equation \eqref{OurTransformDefinition} are important since they resemble a ``non-divergent" case of the Riesz-type potentials associated to $k$-planes. The divergent case, given by
		\begin{equation}
			\label{DivergentRieszPotentialKPlane}
			T_{\alpha}f \left( \pi, x \right) := \int\limits_{\pi} f \left( x - y \right) \| y \|^{\alpha - k} \mathrm{d}y,
		\end{equation}
		where, $\pi \in G_{n, k}$ and $x \in \mathbb{R}^n$, was studied by Duaondikoetxea in \cite{DuoandikoetxeaKPlane}, \cite{DuaondikoetxeaXRayPotentials} and \cite{DuaondikoetxeaPotentialRadon}. The authors of these articles were able to obtain certain mixed norm estimates for the transforms $T_{\alpha}$. While these results are obtained for the general case when $k = 1$ (X-Ray) or $k = n - 1$ (Radon), for the case $1 < k < n - 1$, mixed norm estimates are known only for radial functions. Indeed, obtaining such estimates in the general case is difficult for the $k$-plane transform itself. For a survey on this topic, we refer the reader to \cite{DuoandikoetxeaSurvey}. Motivated from the work in \cite{DuoandikoetxeaKPlane}, Kumar and Ray in \cite{KumarRayWtEst} obtained weighted mixed norm inequalities for the transforms $T_{\alpha}$ acting on radial functions.
		
		We would like to emphasize here that the transforms $R_{\alpha}$ can be interpreted in two ways: One, as the non-divergent case of the operators $T_{\alpha}$, of Riesz-type potentials associated to $k$-planes. Studying such operators is of independent interest and may help in analysing various functions spaces. Secondly, the operator $R_{\alpha}$ given in Equation \eqref{OurTransformDefinition} can be thought of as a weighted $k$-plane transform, where the weight is given by powers of distances from the ``foot of the perpendicular" of the $k$-plane. A physical interpretation of such weights is again given in terms of attenuation or amplification of beams of particles/X-Rays, where the attenuation (amplification) is polynomial (radial power), in contrast to the exponential attenuation (amplification) of Equations \eqref{AttenuatedXRayDefinition} and \eqref{ExponentialXRayDefinition}. In the case when $\alpha < k$, the foot of the perpendicular $u \in \pi^{\perp}$ can be thought of as a singularity (source) of particles\footnote{Here, we write the term ``particles" to mean a general term rather than always considering massive particles. In certain context, the source can be that of massless particles (such as photons) and the interpretation is then to be done in terms of the energy or intensity.}, and the particles emitted by the source interact with the field and lose their mass (intensity) as they travel away from the singularity. When $\alpha > k$, the transforms can be interpreted as the foot of the perpendicular again being a source of certain particles that interact with the field and gain mass as they move away from their source.
		
		Motivated from the above discussion, we shall consider the operators $R_{\alpha}$, and obtain its $L^p$-$L^q$ boundedness. Since obtaining results for general $L^p$-functions is difficult, as discussed previously, we shall focus on obtaining such results for radial functions. As a next step, we obtain $L^p$-$L^p$ boundedness of the operator $R_{\alpha}$ and also evaluate its operator norm. The article is organised as follows: Section \ref{TransfosmPreliminariesSection} deals with certain definitions and preparatory results about the weighted $k$-plane transform and its dual. Next, in Section \ref{ExistenceSection}, we prove the existence of the weighted $k$-plane transform on weighted Lebesgue spaces. Here, we obtain optimal conditions for existence and provide certain examples and counter-examples to support our claim. Then, we move on to prove weighted $L^p$-improving boundedness of $R_{\alpha}$ when it acts on radial functions, in Section \ref{LPImprovingSection}. Next, in Section \ref{EndPointEstimateSection}, we consider the end-point of existence of the transforms $R_{\alpha}$, and prove a Lorentz-space estimate. Both in Section \ref{LPImprovingSection} and Section \ref{EndPointEstimateSection}, we obtain necessary and sufficient conditions for boundedness. Moving on to functions that might not necessarily be radial, we prove weighted $L^p$-$L^p$ boundedness of $R_{\alpha}$, and also evaluate its operator norm in Section \ref{LPLPBoundednessSection}. Lastly, we conclude our article in Section \ref{ConclusionSection} with some remarks and future directions.
		
	\section{The power-weighted \texorpdfstring{$k$}{k}-plane transform and its dual}
		\label{TransfosmPreliminariesSection}
		In this section, we develop certain basic properties of the transforms $R_{\alpha}$ and its (formal) dual that are used throughout the article. We begin by obtaining a description of the formal dual of $R_{\alpha}$.
		\begin{definition}[Formal Dual]
			\label{FormalDualDefiinition}
			The formal dual of the transform $R_{\alpha}$ given in Equation \eqref{OurTransformDefinition} is an operator $R_{\alpha}^*$ that takes functions on the affine Grassmannian and gives out functions on $\mathbb{R}^n$ and satisfies
			\begin{equation}
				\label{DualityEquation}
				\int\limits_{G \left( n, k \right)} R_{\alpha}f \left( \xi \right) \varphi \left( \xi \right) \mathrm{d}\xi = \int\limits_{\mathbb{R}^n} f \left( x \right) R_{\alpha}^* \varphi \left( x \right) \mathrm{d}x,
			\end{equation}
			whenever both the sides of the equation are finite in the Lebesgue sense.
		\end{definition}
		\begin{proposition}
			\label{DualTransform}
			Let $\varphi: G \left( n, k \right) \rightarrow \mathbb{C}$ be a ``nice" function. Then, we have,
			\begin{equation}
				\label{DualTransformEquation}
				R_{\alpha}^* \varphi \left( x \right) = \int\limits_{G_{n,  k}} \varphi \left( \pi, \text{Pr}_{\pi^{\perp}} \left( x \right) \right) \| \text{Pr}_{\pi} \left( x \right) \|^{\alpha - k} \mathrm{d}\pi,
			\end{equation}
			where, $\text{Pr}_{\pi^{\perp}}$ is the orthogonal projection onto $\pi^{\perp}$, $\text{Pr}_{\pi}$ is the orthogonal projection onto $\pi$, and $\mathrm{d}\pi$ is the normalized measure on the compact Grassmannian $G_{n, k}$, of $k$-dimensional subspaces of $\mathbb{R}^n$.
		\end{proposition}
		\begin{proof}
			We have, for all ``nice" functions $f: \mathbb{R}^n \rightarrow \mathbb{C}$, by Equation \eqref{OurTransformDefinition},
			\begin{align*}
				\int\limits_{G \left( n, k \right)} R_{\alpha}f \left( \pi, u \right) \varphi \left( \pi, u \right) \mathrm{d}u \mathrm{d}\pi &= \int\limits_{G_{n, k}} \int\limits_{\pi^{\perp}} \int\limits_{\pi} f \left( x + u \right) \| x \|^{\alpha - k} \varphi \left( \pi, u \right) \mathrm{d}x \mathrm{d}u \mathrm{d}\pi.
			\end{align*}
			We notice that as $x \in \pi$ and $u \in \pi^{\perp}$ varies, $x + u \in \mathbb{R}^n$ varies. In fact, from the inner product structure of $\mathbb{R}^n$, it is clear at once that any point $z \in \mathbb{R}^n$ can be uniquely written as $z = x + u$ for $x \in \pi$ and $u \in \pi^{\perp}$. In this case, it is also clear that $x = \text{Pr}_{\pi} \left( z \right)$ and $u = \text{Pr}_{\pi^{\perp}} \left( z \right)$. Hence, we get
			\begin{align*}
				&\int\limits_{G \left( n, k \right)} R_{\alpha}f \left( \pi, u \right) \varphi \left( \pi, u \right) \mathrm{d}u \mathrm{d}\pi \\
				&= \int\limits_{G_{n, k}} \int\limits_{\mathbb{R}^n} f \left( z \right) \| \text{Pr}_{\pi} \left( z \right) \|^{\alpha - k} \varphi \left( \pi, \text{Pr}_{\pi^{\perp}} \left( z \right) \right) \mathrm{d}z \mathrm{d}\pi.
			\end{align*}
			By using Fubini's theorem, we get
			\begin{equation}
				\label{DualityEquation2}
				\int\limits_{G \left( n, k \right)} R_{\alpha}f \left( \pi, u \right) \varphi \left( \pi, u \right) = \int\limits_{\mathbb{R}^n} f \left( z \right) \left( \int\limits_{G_{n, k}} \varphi \left( \pi, \text{Pr}_{\pi^{\perp}} \left( z \right) \right) \| \text{Pr}_{\pi} \left( z \right) \|^{\alpha - k} \mathrm{d}\pi \right) \mathrm{d}z.
			\end{equation}
			This gives Equation \eqref{DualTransformEquation}.
		\end{proof}
		\begin{remark}
			\normalfont
			We remark here that the quantity $\| \text{Pr}_{\pi} \left( x \right) \|$ gives the distance of $x$ from $\pi^{\perp}$. Therefore, in some sense, the dual $R_{\alpha}^*$ is a power-weighted analogue of the dual $k$-plane transform. We also notice that the set of singularities of $\| \text{Pr}_{\pi} \left( x \right) \|^{\alpha - k}$ (when $\alpha < k$) is a measure zero set.
		\end{remark}
		In this article, we focus our attention to the action of the transform $R_{\alpha}$ on radial functions. For our use in the sequel, we now give formula for $R_{\alpha}f$ and $R_{\alpha}^*\varphi$, when $f$ and $\varphi$ are radial. Indeed, in these cases, the transforms are reduced to one-dimensional transforms and resemble fractional integrals.
		\begin{proposition}
			\label{RadialFormula}
			Let $f: \mathbb{R}^n \rightarrow \mathbb{C}$ be a ``nice" radial function, i.e., there is a function $\tilde{f}: \left[ 0, \infty \right) \rightarrow \mathbb{C}$ such that $f \left( x \right) = \tilde{f} \left( \| x \| \right)$. Then,
			\begin{equation}
				\label{RadialFormulaEquation}
				R_{\alpha}f \left( \pi, u \right) = \left| \mathbb{S}^{k - 1} \right| \int\limits_{\| u \|}^{\infty} \tilde{f} \left( t \right) \left( t^2 - \| u \|^2 \right)^{\frac{\alpha}{2} - 1} t \ \mathrm{d}t.
			\end{equation}
		\end{proposition}
		\begin{proof}
			By the definition of $R_{\alpha}$ (see Equation \eqref{OurTransformDefinition}), we have,
			\begin{align*}
				R_{\alpha}f \left( \pi, u \right) = \int\limits_{\pi} f \left( x + u \right) \| x \|^{\alpha - k} \mathrm{d}x = \int\limits_{\pi} \tilde{f} \left( \| x + u \| \right) \| x \|^{\alpha - k} \mathrm{d}x.
			\end{align*}
			Since $x \perp u$, we have $\| x + u \| = \sqrt{\| x \|^2 + \| u \|^2}$. Further, using the polar decomposition on $\pi$ and observing that the integrand is radial, we obtain
			\begin{align*}
				R_{\alpha}f \left( \pi, u \right) = \left| \mathbb{S}^{k - 1} \right| \int\limits_{0}^{\infty} \tilde{f} \left( \sqrt{r^2 + \| u \|^2} \right) r^{\alpha - 1} \mathrm{d}r.
			\end{align*}
			Equation \eqref{RadialFormulaEquation} is obtained immediately by substituting $r^2 + \| u \|^2 = t^2$.
		\end{proof}
		Using Theorem \ref{RadialFormula}, we now explicitly compute the weighted $k$-plane transform of certain radial functions on $\mathbb{R}^n$. These computations are used in Section \ref{LPImprovingSection}.
		\begin{proposition}
			\label{WtKPlaneBall}
			Let $\alpha > 0$ and $1 \leq k \leq  n - 1$ be fixed. Then, for $\lambda > 0$, we have,
			\begin{equation}
				\label{WtKPlaneBallEquation}
				R_{\alpha}\chi_{B \left( 0, \lambda \right)} \left( \pi, u \right) = \frac{\left| \mathbb{S}^{k - 1} \right|}{\alpha} \left( \lambda^2 - \| u \|^2 \right)^{\frac{\alpha}{2}} \chi_{\left( 0, \lambda \right)} \left( \| u \| \right),
			\end{equation}
			where, $B \left( 0, \lambda \right)$ is the Euclidean ball of radius $\lambda$ centered at $0$.
		\end{proposition}
		\begin{proof}
			Using Equation \eqref{RadialFormulaEquation}, we have,
			\begin{align*}
				R_{\alpha}\chi_{B \left( 0, \lambda \right)} \left( \pi, u \right) &= \left| \mathbb{S}^{k - 1} \right| \int\limits_{\| u \|}^{\infty} \chi_{\left( 0, \lambda \right)} \left( t \right) \left( t^2 - \| u \|^2 \right)^{\frac{\alpha}{2} - 1} t \ \mathrm{d}t \\
				&= \frac{\left| \mathbb{S}^{k - 1} \right|}{\alpha} \left( \lambda^2 - \| u \|^2 \right)^{\frac{\alpha}{2}} \chi_{\left( 0, \lambda \right)} \left( \| u \| \right).
			\end{align*}
		\end{proof}
		\begin{proposition}
			\label{WtKPlaneAnnulus}
			Let $\alpha > 0$ and $1 \leq k \leq n - 1$ be fixed. For $0 < a < b < \infty$, consider the function $f_{a, b}: \mathbb{R}^n \rightarrow \mathbb{C}$, given by $f_{a, b} \left( x \right) = \chi_{\left( a, b \right)} \left( \| x \| \right)$. Then,
			\begin{equation}
				\label{WtKPlaneAnnulusEquation}
				\begin{aligned}
					&R_{\alpha}f \left( \pi, u \right) \\
					&= \frac{\left| \mathbb{S}^{k - 1} \right|}{\alpha} \left[ \left( b^2 - \| u \|^2 \right)^{\frac{\alpha}{2}} - \left( \left( \max \left\lbrace a, \| u \| \right\rbrace \right)^2 - \| u \|^2 \right)^{\frac{\alpha}{2}} \right] \chi_{\left( 0, b \right)} \left( \| u \| \right).
				\end{aligned}
			\end{equation}
		\end{proposition}
		\begin{proof}
			Using Equation \eqref{RadialFormulaEquation}, we get
			\begin{align*}
				&R_{\alpha}f \left( \pi, u \right) = \left| \mathbb{S}^{k - 1} \right| \int\limits_{\| u \|}^{b} \chi_{\left( a, b \right)} \left( t \right) \left( t^2 - \| u \|^2 \right)^{\frac{\alpha}{2} - 1} t \ \mathrm{d}t \\
				&= \frac{\left| \mathbb{S}^{k - 1} \right|}{\alpha} \left[ \left( b^2 - \| u \|^2 \right)^{\frac{\alpha}{2}} - \left( \left( \max \left\lbrace a, \| u \| \right\rbrace \right)^2 - \| u \|^2 \right)^{\frac{\alpha}{2}} \right] \chi_{\left( 0, b \right)} \left( \| u \| \right).
			\end{align*}
		\end{proof}
		We now give the formula for the dual transform $R_{\alpha}^*$ acting on radial functions on $G \left( n, k \right)$.
		\begin{proposition}
			\label{DualRadialFormula}
			Let $\varphi: G \left( n, k \right) \rightarrow \mathbb{C}$ be a ``nice" radial function, i.e., there is a function $\tilde{\varphi}: \left[ 0, \infty \right) \rightarrow \mathbb{C}$ such that $\varphi \left( \xi \right) = \tilde{\varphi} \left( d \left( 0, \xi \right) \right)$. Then,
			\begin{equation}
				\label{DualRadialFornulaEquation}
				R_{\alpha}^*\varphi \left( x \right) = \frac{\left| \mathbb{S}^{k - 1} \right| \left| \mathbb{S}^{n - k - 1} \right|}{\left| \mathbb{S}^{n - 1} \right|} \frac{1}{\| x \|^{n - 2}} \int\limits_{0}^{\| x \|} \tilde{\varphi} \left( s \right) \left( \| x \|^2 - s^2 \right)^{\frac{\alpha}{2} - 1} s^{n - k - 1} \mathrm{d}s.
			\end{equation}
		\end{proposition}
		\begin{proof}
			Since we consider the normalized probability measure on $G_{n, k}$, we can convert the integrals on $G_{n, k}$ to those on $SO \left( n \right)$, without giving any mutliplicative constants. Particularly, we have from Equation \eqref{DualTransformEquation},
			\begin{align*}
				R_{\alpha}^*\varphi \left( x \right) &= \int\limits_{G_{n, k}} \varphi \left( \pi, \text{Pr}_{\pi}^{\perp} \left( x \right) \right) \| \text{Pr}_{\pi} \left( x \right) \|^{\alpha - k} \mathrm{d}\pi \\
				&= \int\limits_{SO \left( n \right)} \varphi \left( g \cdot \mathbb{R}^k, \text{Pr}_{g \cdot \mathbb{R}^{n - k}} \left( x \right) \right) \| \text{Pr}_{g \cdot \mathbb{R}^k} \left( x \right) \|^{\alpha - k} \mathrm{d}g.
			\end{align*}
			Here, $\mathbb{R}^k = \left\lbrace \left( x_1, \cdots, x_k, 0, \cdots, 0 \right) \in \mathbb{R}^n | x_1, \cdots, x_k \in \mathbb{R} \right\rbrace$ and $\mathbb{R}^{n - k} = \left( \mathbb{R}^k \right)^{\perp}$.
			
			If $\xi = \left( g \cdot \mathbb{R}^k, \text{Pr}_{g \cdot \mathbb{R}^{n - k}} \left( x \right) \right)$, then $d \left( 0, \xi \right) = \| \text{Pr}_{g \cdot \mathbb{R}^{n - k}} \left( x \right) \| = \| \text{Pr}_{\mathbb{R}^{n - k}} \left( gx \right) \|$. Therefore, whenever $\varphi$ is radial, we have,
			\begin{align*}
				R_{\alpha}^*\varphi \left( x \right) = \int\limits_{SO \left( n \right)} \tilde{\varphi} \left( \| \text{Pr}_{\mathbb{R}^{n - k}} \left( gx \right) \| \right) \| \text{Pr}_{\mathbb{R}^k} \left( gx \right) \|^{\alpha - k} \mathrm{d}g.
			\end{align*}
			We now notice that as $g \in SO \left( n \right)$, $g \left( \frac{x}{\| x \|} \right)$ covers all of $\mathbb{S}^{n - 1}$, whenever $x \neq 0$. Thus, we obtain,
			\begin{align*}
				R_{\alpha}^*\varphi \left( x \right) = \frac{1}{\left| \mathbb{S}^{n - 1} \right|} \int\limits_{\mathbb{S}^{n - 1}} \tilde{\varphi} \left( \| \text{Pr}_{\mathbb{R}^{n - k}} \left( \| x \| \omega \right) \| \right) \| \text{Pr}_{\mathbb{R}^k} \left( \| x \| \omega \right) \|^{\alpha - k} \mathrm{d}\omega.
			\end{align*}
			Using the bispherical decomposition on $\mathbb{S}^{n - 1}$ (see \cite{GrfakosCFA}), we get
			\begin{align*}
				&R_{\alpha}^*\varphi \left( x \right) \\
				&= \frac{1}{\left| \mathbb{S}^{n - 1} \right|} \int\limits_{\mathbb{S}^{k - 1}} \int\limits_{\mathbb{S}^{n - k - 1}} \int\limits_{0}^{\frac{\pi}{2}} \tilde{\varphi} \left( \| x \| \sin \theta \right) \left( \| x \| \cos \theta \right)^{\alpha - k} \sin^{n - k - 1} \theta \cos^{k - 1} \theta \ \mathrm{d}\theta \mathrm{d}\zeta \mathrm{d}\eta \\
				&= \frac{\left| \mathbb{S}^{k - 1} \right| \left| \mathbb{S}^{n - k - 1} \right|}{\left| \mathbb{S}^{n - 1} \right|} \| x \|^{\alpha - k} \int\limits_{0}^{\frac{\pi}{2}} \tilde{\varphi} \left( \| x \| \sin \theta \right) \sin^{n - k - 1} \theta \cos^{\alpha - 1} \theta \ \mathrm{d}\theta.
			\end{align*}
			Equation \eqref{DualRadialFornulaEquation} is now obtained by substituting $\| x \| \sin \theta = s$.
		\end{proof}
		Using the formula for the dual transform obtained in Proposition \ref{DualRadialFormula}, we give the dual transforms of certain radial functions. These will be utilised in proving the existence of the transform $R_{\alpha}$ on weighted Lebesgue spaces.
		\begin{proposition}
			\label{ComputationDualTransformRadial}
			Let $\varphi: G \left( n, k \right) \rightarrow \mathbb{C}$ be given by $\varphi \left( \xi \right) = \left( d \left( 0, \xi \right) \right)^{\gamma}$, for $\gamma > k - n$. Then,
			\begin{equation}
				\label{DualRadialExample1}
				R_{\alpha}^*\varphi \left( x \right) = \frac{\left| \mathbb{S}^{k - 1} \right| \left| \mathbb{S}^{n - k - 1} \right|}{\left| \mathbb{S}^{n - 1} \right|} \frac{\Gamma \left( \frac{\gamma + n - k}{2} \right) \Gamma \left( \frac{\alpha}{2} \right)}{\Gamma \left( \frac{\gamma + \alpha + n - k}{2} \right)} \| x \|^{\gamma + \alpha - k},
			\end{equation}
			where, $\Gamma \left( \cdot \right)$ is the Gamma function.
		\end{proposition}
		\begin{proof}
			Using Equation \eqref{DualRadialFornulaEquation}, we get
			\begin{align*}
				R_{\alpha}^*\varphi \left( x \right) = \frac{\left| \mathbb{S}^{k - 1} \right| \left| \mathbb{S}^{n - k - 1} \right|}{\left| \mathbb{S}^{n - 1} \right|} \frac{1}{\| x \|^{n - 2}} \int\limits_{0}^{\| x \|} s^{\gamma + n - k - 1} \left( \| x \|^2 - s^2 \right)^{\frac{\alpha}{2} - 1} \mathrm{d}s.
			\end{align*}
			Clearly, the above integral is not convergent if $\gamma \leq k - n$, thereby justifying our assumption. By substituting $s^2 = \| x \|^2 t$ and simplifying, we get
			\begin{align*}
				R_{\alpha}^*\varphi \left( x \right) = \frac{\left| \mathbb{S}^{k - 1} \right| \left| \mathbb{S}^{n - k - 1} \right|}{\left| \mathbb{S}^{n - 1} \right|} \frac{\Gamma \left( \frac{\gamma + n - k}{2} \right) \Gamma \left( \frac{\alpha}{2} \right)}{\Gamma \left( \frac{\gamma + \alpha + n - k}{2} \right)} \| x \|^{\gamma + \alpha - k}.
			\end{align*}
			In the last step, we have used the Beta integral and its relation to Gamma functions (see \cite{Olver}).
		\end{proof}
		\begin{proposition}
			\label{ComputationDualTransformRadial2}
			Let $\varphi: G \left( n, k \right) \rightarrow \mathbb{C}$ be given by 
			$$\varphi \left( \xi \right) = \left( d \left( 0, \xi \right) \right)^{\gamma_1} \left( 1 + \left( d \left( 0, \xi \right) \right)^2 \right)^{\frac{\gamma_2}{2}},$$
			with $\gamma_1 > k - n$ and $\gamma_2 \in \mathbb{R}$. Then,
			\begin{equation}
				\label{DualRadialExample2}
				\begin{aligned}
					&R_{\alpha}^*\varphi \left( x \right) = \frac{\left| \mathbb{S}^{k - 1} \right| \left| \mathbb{S}^{n - k - 1} \right|}{\left| \mathbb{S}^{n - 1} \right|} \frac{\Gamma \left( \frac{\gamma + n - k}{2} \right) \Gamma \left( \frac{\alpha}{2} \right)}{\Gamma \left( \frac{\gamma + \alpha + n - k}{2} \right)} \times \\
					& \| x \|^{\gamma_1 + \alpha - k} \left( 1 + \| x \|^2 \right)^{\frac{\gamma_2}{2}} {}_2F_1 \left( - \frac{\gamma_2}{2}, \frac{\alpha}{2}; \frac{\gamma_1 + n + \alpha - k}{2}; \frac{\| x \|^2}{1 + \| x \|^2} \right),
				\end{aligned}
			\end{equation}
			where, ${}_2F_1 \left( a, b; c; z \right)$ is the Gauss' hypergeometric function.
		\end{proposition}
		\begin{proof}
			Using Equation \eqref{DualRadialFornulaEquation}, we get
			\begin{align*}
				R_{\alpha}^*\varphi \left( x \right) = \frac{\left| \mathbb{S}^{k - 1} \right| \left| \mathbb{S}^{n - k - 1} \right|}{\left| \mathbb{S}^{n - 1} \right|} \frac{1}{\| x \|^{n - 2}} \int\limits_{0}^{\| x \|} s^{\gamma_1 + n - k - 1} \left( 1 + s^2 \right)^{\frac{\gamma_2}{2}} \left( \| x \|^2 - s^2 \right)^{\frac{\alpha}{2} - 1} \mathrm{d}s.
			\end{align*}
			By substituting $s^2 = \| x \|^2 t$ and simplifying, we get
			\begin{align*}
				R_{\alpha}^*\varphi \left( x \right) &= \frac{\left| \mathbb{S}^{k - 1} \right| \left| \mathbb{S}^{n - k - 1} \right|}{\left| \mathbb{S}^{n - 1} \right|} \| x \|^{\gamma_1 + \alpha - k} \int\limits_{0}^{1} \frac{t^{\frac{\gamma_1 + n - k}{2} - 1} \left( 1 - t \right)^{\frac{\alpha}{2} - 1}}{\left( 1 + \| x \|^2 t \right)^{- \frac{\gamma_2}{2}}} \mathrm{d}t \\
				&= \frac{\left| \mathbb{S}^{k - 1} \right| \left| \mathbb{S}^{n - k - 1} \right|}{\left| \mathbb{S}^{n - 1} \right|} \frac{\Gamma \left( \frac{\gamma_1 + n - k}{2} \right) \Gamma \left( \frac{\alpha}{2} \right)}{\Gamma \left( \frac{\gamma_1 + \alpha + n - k}{2} \right)} \| x \|^{\gamma_1 + \alpha - k} \times \\
				&{}_2F_1 \left( - \frac{\gamma_2}{2}, \frac{\gamma_1 + n - k}{2}; \frac{\gamma_1 + \alpha + n - k}{2}; - \| x \|^2 \right).
			\end{align*}
			In the last step, we have used Equation 15.6.1 of \cite{Olver}. Further, by an application of Equation 15.8.1 of \cite{Olver}, we obtain Equation \eqref{DualRadialExample2}.
		\end{proof}
	\section{Existence on Weighted Lebesgue Spaces}
		\label{ExistenceSection}
		In this section we discuss the existence of the operator $R_{\alpha}$ as the integral given in Equation \eqref{OurTransformDefinition} on weighted $L^p$-spaces. Here, we shall consider the weights to be radial powers of the distance. That is,
		$$L^p_{\mu} \left( \mathbb{R}^n \right) := \left\lbrace f: \mathbb{R}^n \rightarrow \mathbb{C} \Bigg| f \text{ is measurable } and \int\limits_{\mathbb{R}^n} \left| f \left(  x \right) \right|^p \| x \|^{\mu} \mathrm{d}x < +\infty \right\rbrace.$$
		Here, we only consider $1 \leq p < \infty$, since at $p = \infty$, the spaces $L^{\infty}_{\mu} \left( \mathbb{R}^n \right)$ and $L^{\infty} \left( \mathbb{R}^n \right)$ coincide.
		
		We now give the result on existence of the weighted $k$-plane transform on the spaces $L^p_{\mu} \left( \mathbb{R}^n \right)$.
		\begin{theorem}
			\label{ExistenceTheorem}
			Let $\alpha > 0$ and $1 \leq k \leq n - 1$ be fixed. For $\mu > \alpha - n$, the operator $R_{\alpha}$ is well-defined on the weighted spaces $L^p_{\mu} \left( \mathbb{R}^n \right)$, for all $1 \leq p < \frac{\mu + n}{\alpha}$. Further, for $\mu = \alpha - n$, the transform $R_{\alpha}$ is well-defined on $L^1_{\mu} \left( \mathbb{R}^n \right)$. Moreover, the conditions obtained on $\mu$ and $p$ are optimal in the sense that if these conditions are not satisfied, there are radial $L^p_{\mu}$-functions $f$, for which $R_{\alpha}f \equiv \infty$.
		\end{theorem}
		\begin{proof}
			To prove the existence we shall employ duality of $R_{\alpha}$ and $R_{\alpha}^*$. We first consider the case $1 < p < \frac{\mu + n}{\alpha}$ and $\mu > \alpha - n$. For the same, let us define $\varphi: G \left( n, k \right) \rightarrow \mathbb{C}$ as
			\begin{equation}
				\label{PhiExistence}
				\varphi \left( \xi \right) = \frac{\left( d \left( 0, \xi \right) \right)^{\mu + \gamma_1}}{\left( 1 + \left( d \left( 0, \xi \right) \right)^2 \right)^{\frac{\gamma_2}{2}}},
			\end{equation}
			where, $\gamma_1 > k - n - \mu$, $\gamma_1 > - \frac{\mu + n}{p'} - \alpha + k$ and $k - n - \mu < \gamma_1 - \gamma_2 < k - \alpha - \frac{\mu + n}{p'}$. Clearly, such a choice is possible since $p < \frac{\mu + n}{\alpha}$. Then, using Equations \eqref{DualityEquation} and \eqref{DualRadialExample2}, we get
			\begin{align*}
				&\left| \int\limits_{G \left( n, k \right)} R_{\alpha}f \left( \xi \right) \varphi \left( \xi \right) \mathrm{d}\xi \right| \\
				&\leq C \int\limits_{X} f \left( x \right) \frac{\| x \|^{\gamma_1 + \mu + \alpha - k}}{\left( 1 + \| x \|^2 \right)^{\frac{\gamma_2}{2}}} \left| {}_2F_1 \left( \frac{\gamma_2}{2}, \frac{\alpha}{2}; \frac{\gamma_1 + \mu + \alpha + n - k}{2}; \frac{\| x \|^2}{1 + \| x \|^2} \right) \right| \mathrm{d}x.
			\end{align*}
			Using the H\"{o}lder's inequality and polar decomposition on $\mathbb{R}^n$, we obtain
			\begin{align*}
				&\left| \int\limits_{G \left( n, k \right)} R_{\alpha}f \left( \xi \right) \varphi \left( \xi \right) \mathrm{d}\xi \right| \leq C \| f \|_{L^p_{\mu} \left( \mathbb{R}^n \right)} \times \\
				&\left[ \int\limits_{0}^{\infty} \frac{r^{p' \left( \gamma_1 + \alpha - k \right) + \mu + n - 1}}{\left( 1 + r^2 \right)^{\frac{p' \gamma_2}{2}}} \left| {}_2F_1 \left( \frac{\gamma_2}{2}, \frac{\alpha}{2}; \frac{\gamma_1 + \mu + \alpha + n - k}{2}; \frac{r^2}{1 + r^2} \right) \right|^{p'} \mathrm{d}r \right]^{\frac{1}{p}}.
			\end{align*}
			From Equation 15.4.20 of \cite{Olver}, we know that the hypergeometric function is bounded due to our choice of $\gamma_1$ and $\gamma_2$. Further, due to the choice of $\gamma_1$ and $\gamma_2$, the integral on the right hand side of the above inequality is convergent. Consequently, $R_{\alpha}f$ is finite almost everywhere.
			
			For $p = 1$ and $\mu > \alpha - n$, we consider the following. Again using Equations \eqref{DualityEquation} and \eqref{DualRadialExample1}, we get
			\begin{align*}
				\left| \int\limits_{G \left( n, k \right)} R_{\alpha}f \left( \xi \right) \left( d \left( 0, \xi \right) \right)^{\mu + k - \alpha} \mathrm{d}\xi \right| &\leq C \int\limits_{\mathbb{R}^n} f \left( x \right) \| x \|^{\mu} \mathrm{d}x = C \| f \|_{L^1_{\mu} \left( \mathbb{R}^n \right)}.
			\end{align*}
			Clearly, for $L^1_{\mu} \left( \mathbb{R}^n \right)$ functions, we must have $R_{\alpha}f$ finite almost everywhere.
			
			The proof will be complete once we show the result for $\mu = \alpha - n$ and $p = 1$. For the same, let us consider the function $\varphi: G \left( n, k \right) \rightarrow \mathbb{C}$, given in Equation \eqref{PhiExistence}, with $\gamma_1 > k - \alpha$ and $\gamma_2 > \gamma_1 + \alpha - k$. Then, using Equations \eqref{DualityEquation} and \eqref{DualRadialExample2}, we get
			\begin{align*}
				\left| \int\limits_{G \left( n, k \right)} R_{\alpha}f \left( \xi \right) \varphi \left( \xi \right) \mathrm{d}\xi \right| &\leq C \int\limits_{\mathbb{R}^n} \left| f \left( x \right) \right| \frac{\| x \|^{\mu + \gamma_1 + \alpha - k}}{\left( 1 + \| x \|^2 \right)^{\frac{\gamma_2}{2}}} \times \\
				&\left| {}_2F_1 \left( \frac{\gamma_2}{2}, \frac{\alpha}{2}; \frac{\gamma_1 + \mu + \alpha + n - k}{2}; \frac{\| x \|^2}{1 + \| x \|^2} \right) \right| \mathrm{d}x \\
				&= C \left[ \int\limits_{\| x \| \leq 1} \left| f \left( x \right) \right| \| x \|^{\mu} \frac{\| x \|^{\gamma_1 + \alpha - k}}{\left( 1 + \| x \|^2 \right)^{\frac{\gamma_2}{2}}} \times \right. \\
				&\left. \left| {}_2F_1 \left( \frac{\gamma_2}{2}; \frac{\alpha}{2}; \frac{\gamma_1 + \mu + \alpha + n - k}{2}; \frac{\| x \|^2}{1 + \| x \|^2} \right) \right| \mathrm{d}x \right. \\
				&\left. + \int\limits_{\| x \| > 1} \left| f \left( x \right) \right| \| x \|^{\mu} \frac{\| x \|^{\gamma_1 + \alpha - k}}{\left( 1 + \| x \|^2 \right)^{\frac{\gamma_2}{2}}} \times \right. \\
				&\left. \left| {}_2F_1 \left( \frac{\gamma_2}{2}, \frac{\alpha}{2}; \frac{\gamma_1 + \mu + \alpha + n - k}{2}; \frac{\| x \|^2}{1 + \| x \|^2} \right) \right| \mathrm{d}x \right].
			\end{align*}
			It is clear that for $\| x \| \leq 1$, the term $\frac{\| x \|^{\gamma_1 + \alpha - k}}{\left( 1 + \| x \|^2 \right)^{\frac{\gamma_2}{2}}} \left| {}_2F_1 \left( \frac{\gamma_2}{2}; \frac{\alpha}{2}; \frac{\gamma_1 + \mu + \alpha + n - k}{2}; \frac{\| x \|^2}{1 + \| x \|^2} \right) \right|$ is bounded. On the other hand, by the choice of $\gamma_1$ and $\gamma_2$, we have from Equation 15.4.23 of \cite{Olver} for $\| x \| > 1$,
			$$\left| {}_2F_1 \left( \frac{\gamma_2}{2}, \frac{\alpha}{2}; \frac{\gamma_1 + \mu + \alpha + n - k}{2}; \frac{\| x \|^2}{1 + \| x \|^2} \right) \right| \leq C \frac{1}{\left( 1 + \| x \|^2 \right)^{\frac{\gamma_1 - \gamma_2 + \alpha - k}{2}}}.$$
			Hence, we obtain
			\begin{align*}
				&\left| \int\limits_{G \left( n, k \right)} R_{\alpha}f \left( \xi \right) \varphi \left( \xi \right) \mathrm{d}\xi \right| \\
				&\leq C \left[ \int\limits_{\| x \| \leq 1} \left| f \left( x \right) \right| \| x \|^{\mu} \mathrm{d}x + \int\limits_{\| x \| > 1} \left| f \left( x \right) \right| \| x \|^{\mu} \frac{\| x \|^{\gamma_1 + \alpha - k}}{\left( 1 + \| x \|^2 \right)^{\frac{\gamma_1 + \alpha - k}{2}}} \mathrm{d}x \right] \\
				&\leq C \| f \|_{L^1_{\mu} \left( \mathbb{R}^n \right)}.
			\end{align*}
			Therefore, $R_{\alpha}$ is well-defined on $L^1_{\alpha - n} \left( \mathbb{R}^n \right)$.
			
			The optimality of the conditions on $p$ and $\mu$ is shown in Examples \ref{CounterExample1} and \ref{CounterExample2}.
		\end{proof}
		We now give certain counterexamples to the existence of $R_{\alpha}$ when the conditions on $p$ and $\mu$ as mentioned in Theorem \ref{ExistenceTheorem} are not valid.
		\begin{example}
			\label{CounterExample1}
			\normalfont
			First, let us consider $\mu > \alpha - n$ and $p \geq \frac{\mu + n}{\alpha} > 1$. We define the function $f: \mathbb{R}^n \rightarrow \mathbb{C}$ as
			\begin{equation}
				\label{CounterExample1Function}
				f \left( x \right) = \frac{\| x \|^{- \frac{\mu}{p}} \left( 2 + \| x \| \right)^{- \frac{n}{p}}}{\ln \left( 2 + \| x \| \right)}.
			\end{equation}
			Then, we have, using the polar decomposition on $\mathbb{R}^n$,
			$$\int\limits_{\mathbb{R}^n} \left| f \left( x \right) \right|^p \| x \|^{\mu} \mathrm{d}x = C \int\limits_{0}^{\infty} \frac{\left( 2 + r \right)^{-n} r^{n - 1}}{\ln^{p} \left( 2 + r \right)} \mathrm{d}r.$$
			By substituting $\ln \left( 2 + r \right) = t$, we see that
			$$\int\limits_{\mathbb{R}^n} \left| f \left( x \right) \right|^p \| x \|^{\mu} \mathrm{d}x = C \int\limits_{\ln 2}^{\infty} \frac{1}{t^{p}} \frac{\left( e^t - 2 \right)^{n - 1}}{e^{\left( n - 1 \right)t}} \mathrm{d}t < + \infty,$$
			since $p > 1$. That is, $f \in L^p_{\mu} \left( \mathbb{R}^n \right)$. On the other hand, we have from Equation \eqref{RadialFormulaEquation} for any $\xi = \left( \pi, u \right) \in G \left( n, k \right)$,
			\begin{align*}
				R_{\alpha}f \left( \pi, u \right) &= C \int\limits_{\| u \|}^{\infty} \frac{t^{1 - \frac{\mu}{p}} \left( 2 + t \right)^{- \frac{n}{p}}}{\ln \left( 2 + t \right)} \left( t^2 - \| u \|^2 \right)^{\frac{\alpha}{2} - 1} \mathrm{d}t.
			\end{align*}
			Upon substituting $\ln \left( 2 + t \right) = s$, we get
			\begin{align*}
				R_{\alpha}f \left( \pi, u \right) &- C \int\limits_{\ln \left( 2 + \| s \| \right)}^{\infty} \frac{\left( e^s - 2 \right)^{1 - \frac{\mu}{p}}}{e^{\left( \frac{n}{p} - 1 \right) s}} \left( \left( e^s - 2 \right)^2 - \| u \|^2 \right)^{\frac{\alpha}{2} - 1} \frac{1}{s} \mathrm{d}s.
			\end{align*}
			Clearly, the total power of $e^s$ towards $+\infty$ is $\alpha - \frac{\mu + n}{p} \geq 0$. Therefore, the integral in the above expression is not convergent and we get $R_{\alpha}f \left( \pi, u \right) = + \infty$, for any $\left( \pi, u \right) \in G \left( n, k \right)$.
		\end{example}
		\begin{example}
			\label{CounterExample2}
			\normalfont
			To see the optimality of the condition on $\mu$, we consider the function $f: \mathbb{R}^n \rightarrow \mathbb{C}$, defined by
			\begin{equation}
				\label{CounterExample2Function}
				f \left( x \right) = \frac{\| x \|^{- \frac{\mu}{p}} \left( 2 + \| x \| \right)^{- \frac{n}{p}}}{\ln^{\frac{1}{p} + \delta} \left( 2 + \| x \| \right)},
			\end{equation}
			for some $\delta > 0$. Then, we have, by employing the polar decomposition of $\mathbb{R}^n$,
			\begin{align*}
				\int\limits_{\mathbb{R}^n} \left| f \left( x \right) \right|^p \| x \|^{\mu} \mathrm{d}x = C \int\limits_{0}^{\infty} \frac{\left( 2 + r \right)^{-n} r^{n - 1}}{\ln^{1 + \delta} \left( 2 + r \right)} \mathrm{d}r.
			\end{align*}
			By substituting $\ln \left( 2 + r \right) = t$, we get
			\begin{align*}
				\int\limits_{\mathbb{R}^n} \left| f \left( x \right) \right|^p \| x \|^{\mu} \mathrm{d}x = C \int\limits_{\ln 2}^{\infty} \frac{1}{t^{1 + p \delta}} \frac{\left( e^t - 2 \right)^{n - 1}}{e^{\left( n - 1 \right) t}} \mathrm{d}t < +\infty,
			\end{align*}
			since $\delta > 0$. That is, $f \in L^p_{\mu} \left( \mathbb{R}^n \right)$ for any $p \geq 1$. As for the weighted $k$-plane transform, we have using Equation \eqref{RadialFormulaEquation},
			\begin{align*}
				R_{\alpha}f \left( \pi, u \right) &= C \int\limits_{\| u \|}^{\infty} \frac{t^{1 - \frac{\mu}{p}} \left( 2 + t \right)^{- \frac{n}{p}}}{\ln^{\frac{1}{p} + \delta} \left( 2 + t \right)} \left( t^2 - \| u \|^2 \right)^{\frac{\alpha}{2} - 1} \mathrm{d}t.
			\end{align*}
			Again, we substitute $\ln \left( 2 + t \right) = s$ and obtain
			\begin{align*}
				R_{\alpha}f \left( \pi, u \right) &= C \int\limits_{\ln \left( 2 + \| u \| \right)}^{\infty} \frac{1}{s^{\frac{1}{p} + \delta}} \frac{\left( e^s - 2 \right)^{1 - \frac{\mu}{p}}}{e^{\left( \frac{n}{p} - 1 \right) s}} \left( \left( e^s - 2 \right)^2 - \| u \|^2 \right)^{\frac{\alpha}{2} - 1} \mathrm{d}s.
			\end{align*}
			Clearly, the total power of $e^s$ towards $+\infty$ is $\alpha - \frac{\mu + n}{p} > 0$ for $\mu < \alpha - n$ and $p \geq 1$ or for $\mu = \alpha - n$ and $p > 1$. In either case, the integral is not convergent.
		\end{example}
	\section{\texorpdfstring{$L^p$}{Lp}-improving boundedness for radial functions}
		\label{LPImprovingSection}
		In this section, we consider the problem of weighted $L^p$-$L^q$ boundedness of the transform $R_{\alpha}$ acting on radial functions. We have seen in Theorem \ref{RadialFormula} that the action of $R_{\alpha}$ on radial functions is one-dimensional and therefore, is easier to analyse. For the purpose of our considerations, we require weighted Lebesgue spaces on $G \left( n, k \right)$. They are defined as follows.
		\begin{align*}
			L^p_{\nu} \left( G \left( n, k \right) \right) := \Bigg\lbrace \varphi: G \left( n, k \right) \rightarrow \mathbb{C} \Bigg| &\varphi \text{ is measurable and } \\
			&\int\limits_{G \left( n, k \right)} \left| \varphi \left( \pi, u \right) \right|^p \| u \|^{\nu} \mathrm{d}u \mathrm{d}\pi < +\infty \Bigg\rbrace.
		\end{align*}
		Again, such spaces are considered only for $1 \leq p < \infty$, since $L^{\infty}_{\nu} \left( G \left( n, k \right) \right)$ coincides with $L^{\infty} \left( G \left( n, k \right) \right)$.
		
		Particularly, we shall answer the following questions in this section.
		\begin{question}
			\label{Question1}
			What are the admissible values of $\mu, \nu, p$ and $r$ such that we have the following inequality for all radial functions $f \in L^p_{\mu} \left( \mathbb{R}^n \right)$?
			\begin{equation}
				\label{RequiredInequality1}
				\| R_{\alpha}f \|_{L^r_{\nu} \left( G \left( n, k \right) \right)} \leq C \| f \|_{L^p_{\mu} \left( \mathbb{R}^n \right)}.
			\end{equation}
		\end{question}
		\begin{question}
			\label{Question2}
			What are the admissible values of $\mu$, $\nu$ and $p$ such that for all radial functions $f \in L^p_{\mu} \left( \mathbb{R}^n \right)$, the following inequality holds?
			\begin{equation}
				\label{RequiredInequality2}
				\| \left( d \left( 0, \cdot \right) \right)^{\nu} R_{\alpha}f \left( \cdot \right) \|_{L^{\infty} \left( G \left( n, k \right) \right)} \leq C \| f \|_{L^p_{\mu} \left( \mathbb{R}^n \right)}.
			\end{equation}
		\end{question}
		We begin by obtaining necessary conditions for Inequality \eqref{RequiredInequality1} to hold.
		\begin{theorem}
			\label{NecessaryLPLR}
			For $\alpha > 0$, $\mu > \alpha - n$, and $1 \leq p < \frac{\mu + n}{\alpha}$, the operator $R_{\alpha}$ is bounded from $L^p_{\mu} \left( \mathbb{R}^n \right)$ to $L^r_{\nu} \left( G \left( n, k \right) \right)$ only if
			\begin{equation}
				\label{NecessaryConditionsLPLR}
				\nu > k - n \text{ and } \frac{\mu + n}{p} - \frac{\nu + n - k}{r} = \alpha.
			\end{equation}
		\end{theorem}
		\begin{proof}
			We first notice that $\mu > \alpha - n > -n$. It is then clear that
			\begin{equation}
				\label{BallMass}
				\| \chi_{B \left( 0, \lambda \right)} \|_{L^p_{\mu} \left( \mathbb{R}^n \right)} = C \left[ \int\limits_{0}^{\lambda} t^{\mu + n - 1} \right]^{\frac{1}{p}} = C \lambda^{\frac{\mu + n}{p}}.
			\end{equation}
			On the other hand, using Equation \eqref{WtKPlaneBallEquation}, and polar decomposition on $\pi^{\perp}$, we get
			\begin{align*}
				\| R_{\alpha}\chi_{B \left( 0, \lambda \right)} \|_{L^r_{\nu} \left( G \left( n, k \right) \right)} &= C \left[  \int\limits_{G_{n, k}} \int\limits_{\pi^{\perp}} \left( \lambda^2 - \| u \|^2 \right)^{\frac{r \alpha}{2}} \chi_{\left( 0, \lambda \right)} \left( \| u \| \right) \| u \|^{\nu} \mathrm{d}u \mathrm{d}\pi \right]^{\frac{1}{r}} \\
				&= C \left[ \int\limits_{0}^{\lambda} \left( \lambda^2 - t^2 \right)^{\frac{r \alpha}{2}} t^{\nu + n - k - 1} \mathrm{d}t \right]^{\frac{1}{r}}.
			\end{align*}
			By substituting $t^2 = \lambda^2 y$ and simplifying, we get
			\begin{equation}
				\label{WtKPlaneBallNorm}
				\| R_{\alpha}\chi_{B \left( 0, \lambda \right)} \|_{L^r_{\nu} \left( G \left( n, k \right) \right)} = C \lambda^{\alpha + \frac{\nu + n - k}{r}},
			\end{equation}
			provided $\nu > k - n$. Using Equations \eqref{WtKPlaneBallNorm} and \eqref{BallMass} in Inequality \eqref{RequiredInequality1} and taking $\lambda \rightarrow 0$ and $\lambda \rightarrow \infty$, we get the necessary condition mentioned in Equation \eqref{NecessaryConditionsLPLR}.
		\end{proof}
		We now show the sufficiency of the conditions mentioned in Theorem \ref{NecessaryLPLR}.
		\begin{theorem}
			\label{SufficientLPLR}
			Let $\alpha > 0$, $\mu > \alpha - n$ and $1 \leq p < \frac{\mu + n}{\alpha}$ and $r \geq p$.
			\begin{enumerate}
				\item[\mylabel{AlphaMoreThan2}{(A)}] For $\alpha \geq 2$, Inequality \eqref{RequiredInequality1} holds for all radial functions if and only if $\nu$ satisfies conditions mentioned in Equation \eqref{NecessaryConditionsLPLR}.
				\item[\mylabel{AlphaLessThan2}{(B)}] For $0 < \alpha < 2$, Inequality \eqref{RequiredInequality1} holds for all radial functions if and only if $\nu$ satisfies the conditions of Equation \eqref{NecessaryConditionsLPLR} and $\frac{1}{p} - \frac{1}{r} \leq \frac{\alpha}{2}$.
			\end{enumerate}
		\end{theorem}
		\begin{proof}
			\begin{enumerate}
				\item[\mylabel{AlphaMoreThan2Proof}{(A)}] By using Equations \eqref{RadialFormulaEquation} and \eqref{NecessaryConditionsLPLR} and the polar decomposition on $\pi^{\perp}$, we get
				\begin{align}
					&\| R_{\alpha}f \|_{L^r_{\nu} \left( G \left( n, k \right) \right)} \nonumber \\
					&= C \left[ \int\limits_{0}^{\infty} \left| \int\limits_{s}^{\infty} \tilde{f} \left( t \right) \left( t^2 - s^2 \right)^{\frac{\alpha}{2} - 1} t \ \mathrm{d}t \right|^r s^{\nu + n - k - 1} \mathrm{d}s \right]^{\frac{1}{r}} \nonumber \\
					\label{ConvolutionDescription}
					&= C \left[ \int\limits_{0}^{\infty} \left| \int\limits_{s}^{\infty} \tilde{f} \left( t \right) t^\frac{\mu + n}{p} \left( 1 - \frac{s^2}{t^2} \right)^{\frac{\alpha}{2} - 1} \left( \frac{s}{t} \right)^{\frac{\nu + n - k}{r}} \frac{\mathrm{d}t}{t} \right|^r \frac{\mathrm{d}s}{s} \right]^{\frac{1}{r}} \\
					\label{ConvolutionFG}
					&= C \| F * G \|_{L^r \left( \left( 0, \infty \right), \frac{\mathrm{d}t}{t} \right)},
				\end{align}
				where,
				\begin{equation}
					\label{FConvolution}
					F \left( t \right) = \tilde{f} \left( t \right) t^{\frac{\mu + n}{p}},
				\end{equation}
				and
				\begin{equation}
					\label{GConvolution}
					G \left( t \right) = \left( 1 - t^2 \right)^{\frac{\alpha}{2} - 1} t^{\frac{\nu + n - k}{r}} \chi_{\left( 0, 1 \right)} \left( t \right).
				\end{equation}
				It is clear that $\| F \|_{L^p \left( \left( 0, \infty \right) \right)} = C \| f \|_{L^p_{\mu} \left( \mathbb{R}^n \right)}$, and for $\alpha \geq 2$ and $\nu > k - n$, $G \in L^q \left( \left( 0, \infty \right), \frac{\mathrm{d}t}{t} \right)$. Therefore, using Young's convolution inequality on Equation \eqref{ConvolutionFG}, we get
				$$\| R_{\alpha}f \|_{L^r_{\nu} \left( G \left( n, k \right) \right)} \leq C \| f \|_{L^p_{\mu} \left( \mathbb{R}^n \right)}.$$
				\item[\mylabel{AlphaLessThan2Proof}{(B)}] The proof for the case when $\frac{1}{p} - \frac{1}{r} < \frac{\alpha}{2}$ is verbatim to \ref{AlphaMoreThan2Proof}. We consider the boundary case $\frac{1}{p} - \frac{1}{r} = \frac{\alpha}{2}$. To prove the boundedness in this case, we employ the theory of Riemann-Liouville fractional integrals. For a thorough study on fractional integrals, we refer the reader to \cite{RubinFIP}.
				
				Using Equation \eqref{RadialFormula} and polar decomposition on $\pi^{\perp}$, we obtain
				\begin{align*}
					\| R_{\alpha}f \|_{L^r_{\nu} \left( G \left( n, k \right) \right)} = C \left[ \int\limits_{0}^{\infty} \left| \int\limits_{s}^{\infty} \tilde{f} \left( t \right) \left( t^2 - s^2 \right)^{\frac{\alpha}{2} - 1} t \ \mathrm{d}t \right| s^{\nu + n - k - 1} \mathrm{d}s \right]^{\frac{1}{r}}.
				\end{align*}
				Next, we substitute $t^2 = u$ and $s^2 = v$ and get
				\begin{equation}
					\label{RAlphaFINorm}
					\begin{aligned}
						\| R_{\alpha}f \|_{L^r_{\nu} \left( G \left( n, k \right) \right)} &= C \left[ \int\limits_{0}^{\infty} \left| \int\limits_{v}^{\infty} \tilde{f} \left( \sqrt{u} \right) \left( u - v \right)^{\frac{\alpha}{2} - 1} \mathrm{d}u \right|^r v^{\frac{\nu + n - k}{2} - 1} \mathrm{d}v \right]^{\frac{1}{r}} \\
						&= C \| \rho_- \left( \cdot \right) I^{\frac{\alpha}{2}}_-\varphi \|_{L^r \left( \left( 0, \infty \right) \right)},
					\end{aligned}
				\end{equation}
				where, $\rho_- \left( t \right) = t^{\frac{\nu + n - k}{2r} - \frac{1}{r}}$ and $I^{\frac{\alpha}{2}}_-$ is the left-sided Riemann-Liouville fractional integral (for the definition, we refer the reader to \cite{RubinFIP}), and $\varphi \left( t \right) = \tilde{f} \left( \sqrt{t} \right)$. Now, using Theorem 8.13 of \cite{RubinFIP}, we get
				\begin{align*}
					\| R_{\alpha}f \|_{L^r_{\nu} \left( G \left( n, k \right) \right)} &\leq C \left[ \int\limits_{0}^{\infty} \left| t^{\frac{\nu + n - k}{2r} - \frac{1}{r}} \tilde{f} \left( \sqrt{t} \right) \right|^p \mathrm{d}t \right]^{\frac{1}{p}} \\
					&= C \left[ \int\limits_{0}^{\infty} \left| \tilde{f} \left( \sqrt{t} \right) \right|^p t^{\frac{\mu + n}{2} - 1} \mathrm{d}t \right]^{\frac{1}{p}} = C \| f \|_{L^p_{\mu} \left( \mathbb{R}^n \right)}.
				\end{align*}
				Here, we have used the facts that $\frac{\mu + n}{p} - \frac{\nu + n - k}{r} = \alpha$ and $\frac{1}{p} - \frac{1}{r} = \frac{\alpha}{2}$.
			\end{enumerate}
		\end{proof}
		\begin{remark}
			\label{RemarkLPLR}
			\normalfont
			Certain remarks are in order.
			\begin{enumerate}
				\item First, we notice that from Equation \eqref{ConvolutionFG} it is clear that $R_{\alpha}$ has a convolution structure. Therefore, by the famous Theorem due to H\"{o}rmander, we must have $R_{\alpha}$ to be $L^p$-improving (see \cite{Hormander}). That is, $r \geq p$ is necessary.
				\item The conditions on $\mu$ and $p$ are necessary for the existence of the transforms $R_{\alpha}$, as discussed in Theorem \ref{ExistenceTheorem}.
				\item The necessity of the condition $\frac{1}{p} - \frac{1}{r} \leq \frac{\alpha}{2}$ is not immediate. To see the same, we consider the following, motivated from \cite{KumarRayWtEst}. First, using $t = e^y$ and $s = e^x$ in Equation \eqref{ConvolutionDescription}, we obtain
				$$\| R_{\alpha}f \|_{L^r_{\nu} \left( G \left( n, k \right) \right)} = C \| \varphi * \psi \|_{L^r \left( \mathbb{R} \right)},$$
				where, $\varphi \left( x \right) = \tilde{f} \left( e^x \right) e^{\left( \frac{\mu + n}{p} \right)y}$, and 
				$$\psi \left( x \right) = \left( 1 - e^{2x} \right)^{\frac{\alpha}{2} - 1} e^{\left( \frac{\nu + n - k}{r} \right)x} \chi_{\left( - \infty, 0 \right)} \left( x \right).$$
				We notice that the convolution operator $\varphi \mapsto \varphi * \psi$ is a non-negative operator and Inequality \eqref{RequiredInequality1} implies the boundedness of $\varphi \mapsto \varphi * \psi$. For $\varphi \geq 0$ and $\lambda > 0$, let $\varphi_{\lambda} \left( x \right) = \varphi \left( \lambda x \right)$. Then, it is clear at once that for $\lambda > 1$, we have
				\begin{align*}
					\| \varphi_{\lambda} * \psi \|_{L^r \left( \mathbb{R} \right)} &\geq \left[ \int\limits_{\mathbb{R}} \left( \int\limits_{\mathbb{R}} \varphi \left( e^{\lambda x} \right) \psi \left( e^{x - y} \right) \chi_{\left( -1, 0 \right)} \left( y \right) \right)^r \mathrm{d}x \right]^{\frac{1}{r}} \\
					&\geq C \lambda^{- \frac{\alpha}{2} - \frac{1}{r}} \| \varphi * \psi \|_{L^r \left( \mathbb{R} \right)}.
				\end{align*}
				On the other hand, we also have,
				$$\| \varphi_{\lambda} \|_{L^p \left( \mathbb{R} \right)} = \lambda^{- \frac{1}{p}} \| \varphi \|_{L^p \left( \mathbb{R} \right)}.$$
				Since Inequality \eqref{RequiredInequality1} implies the $L^p$-$L^r$ boundedness of $\varphi \mapsto \varphi * \psi$, we must have for $\lambda > 1$, and a fixed $\varphi \geq 0$ with $\varphi \in L^p \left( \mathbb{R} \right)$,
				$$\lambda^{- \frac{\alpha}{2} - \frac{1}{r}} \| \varphi * \psi \|_{L^r \left( \mathbb{R} \right)} \leq C \lambda^{- \frac{1}{p}} \| \varphi \|_{L^p \left( \mathbb{R} \right)}.$$
				As $\lambda \rightarrow \infty$, the above inequality forces $\frac{1}{p} - \frac{1}{r} \leq \frac{\alpha}{2}$.
			\end{enumerate}
		\end{remark}
		Theorem \ref{SufficientLPLR} and Remark \ref{RemarkLPLR} completely answer Question \ref{Question1}. We now move on to answer Question \ref{Question2}. As before, let us first obtain necessary conditions for Inequality \eqref{RequiredInequality2} to hold.
		\begin{theorem}
			\label{NecessaryLPLInfinity}
			Let $\alpha > 0$, $\mu > \alpha - n$ and $1 \leq p < \frac{\mu + n}{\alpha}$. Then, Inequality \eqref{RequiredInequality2} holds only if
			\begin{equation}
				\label{NeceessaryConditionLPLInfinity}
				\frac{\mu + n}{p} = \nu + \alpha.
			\end{equation}
		\end{theorem}
		\begin{proof}
			From Equation \eqref{WtKPlaneBallEquation}, it is clear, by choosing $\| u \| = \frac{\lambda}{2}$, that
			\begin{equation}
				\label{LInfinityEstimaeteBall}
				\| \left( d \left( 0, \cdot \right) \right)^{\nu} R_{\alpha}\chi_{B \left( 0, \lambda \right)} \left( \cdot \right) \|_{L^{\infty} \left( G \left( n, k \right) \right)} \geq C \lambda^{\nu + \alpha}.
			\end{equation}
			Using Inequality \eqref{LInfinityEstimaeteBall} and Equation \eqref{BallMass} in Inequality \eqref{RequiredInequality2}, we have for every $\lambda > 0$,
			$$\lambda^{\nu + \alpha} \leq C \lambda^{\frac{\mu + n}{p}}.$$
			The necessary condition mentioned in Equation \eqref{NeceessaryConditionLPLInfinity} is obtained by taking $\lambda \rightarrow 0$ and $\lambda \rightarrow \infty$.
		\end{proof}
		We now prove the sufficiency of the condition mentioned in Theorem \ref{NecessaryLPLInfinity}.
		\begin{theorem}
			\label{SufficientLPLInfinity}
			Let $\alpha > 0$, $\mu > \alpha - n$, and $1 \leq p < \frac{\mu + n}{\alpha}$.
			\begin{enumerate}
				\item[\mylabel{AlphaMoreThan2Infinity}{(A)}] For $\alpha \geq 2$, Inequality \eqref{RequiredInequality2} holds for all radial functions in $L^p_{\mu} \left( \mathbb{R}^n \right)$ if and only if $\nu$ satisfies Equation \eqref{NeceessaryConditionLPLInfinity}.
				\item[\mylabel{AlphaLessThan2Infinity}{(B)}] For $0 < \alpha < 2$, Inequality \eqref{RequiredInequality2} holds for all radial functions in $L^p_{\mu} \left( \mathbb{R}^n \right)$ if and only if $\nu$ satisfies Equation \eqref{NeceessaryConditionLPLInfinity} and $p > \frac{2}{\alpha}$.
			\end{enumerate}
		\end{theorem}
		\begin{proof}~
			\begin{enumerate}
				\item[\mylabel{AlphaMoreThan2InfinityProof}{(A)}] We consider two cases.
				
				\textbf{Case I:} First, let us consider $p > 1$. We have, using Equation \eqref{RadialFormulaEquation} and Equation \eqref{NeceessaryConditionLPLInfinity},
				\begin{equation}	
					\label{LPLInfinityStart}
					\| u \|^{\nu} \left| R_{\alpha}f \left( \pi, u \right) \right| \leq C \| u \|^{\frac{\mu + n}{p} - \alpha} \int\limits_{\| u \|}^{\infty} \left| \tilde{f} \left( t \right) \right| t^{\alpha - 1} \left( 1 - \frac{\| u \|^2}{t^2} \right)^{\frac{\alpha}{2} - 1} \mathrm{d}t.
				\end{equation}
				By an application of H\"{o}lder's inequality, we get
				\begin{align*}
					&\| u \|^{\nu} \left| R_{\alpha}f \left( \pi, u \right) \right| \\
					&\leq C \| u \|^{\frac{\mu + n}{p} - \alpha} \| f \|_{L^p_{\mu} \left( \mathbb{R}^n \right)} \left[ \int\limits_{\| u \|}^{\infty} \left( 1 - \frac{\| u \|^2}{t^2} \right)^{p' \left( \frac{\alpha}{2} - 1 \right)} t^{p' \left( \alpha - \frac{\mu + n}{p} \right) - 1} \mathrm{d}t \right]^{\frac{1}{p'}}.
				\end{align*}
				By replacing $t$ with $\frac{\| u \|}{t}$, we obtain
				\begin{align*}
					\| u \|^{\nu} \left| R_{\alpha}f \left( \pi, u \right) \right| &\leq C \| f \|_{L^p_{\mu} \left( \mathbb{R}^n \right)} \left[ \int\limits_{0}^{1} \left( 1 - t^2 \right)^{p' \left( \frac{\alpha}{2} - 1 \right)} t^{-p' \left( \alpha - \frac{\mu + n}{p} \right) - 1} \mathrm{d}t \right]^{\frac{1}{p'}}.
				\end{align*}
				We notice that the integral on the right side of the above inequality is finite since $p < \frac{\mu + n}{\alpha}$ and $p > 1 \geq \frac{2}{\alpha}$.
				
				\textbf{Case II:} We now consider the case when $p = 1$. In Equation \eqref{LPLInfinityStart}, we notice that $\left( 1 - \frac{\| u \|^2}{t^2} \right)^{\frac{\alpha}{2} - 1} \leq 1$, since $\alpha \geq 2$. Also, for $t > \| u \|$, we get $t^{\mu + n - \alpha} > \| u \|^{\mu + n - \alpha}$. Using this together with Equation \eqref{NeceessaryConditionLPLInfinity}, we get the Inequality \eqref{RequiredInequality2}.
				\item[\mylabel{AlphaLessThan2InfinityProof}{(B)}] First, we notice that $\alpha < 2$ gives $p > \frac{2}{\alpha} > 1$. The proof in this case is verbatim to Case I of \ref{AlphaMoreThan2InfinityProof}. The necessity of $p > \frac{2}{\alpha}$ is discussed in Example \ref{CounterExample3} and Remark \ref{FIRemark}.
			\end{enumerate}
		\end{proof}
		We now give the necessity of the condition $p > \frac{2}{\alpha}$ in part \ref{AlphaLessThan2Infinity} of Theorem \ref{SufficientLPLInfinity}. We begin by showing that $p \geq \frac{2}{\alpha}$ is necessary.
		\begin{example}
			\label{CounterExample3}
			\normalfont
			We consider the function $f_{a, b}: \mathbb{R}^n \rightarrow \mathbb{C}$ given in Proposition \ref{WtKPlaneAnnulus}. By choosing $\| u \| = a$ in Equation \eqref{WtKPlaneAnnulusEquation}, it is clear that
			\begin{equation}
				\label{LInfinityEstimateAnnulus}
				\| \left( d \left( 0, \cdot \right) \right)^{\nu} R_{\alpha}f_{a, b} \left( \cdot \right) \|_{L^{\infty} \left( G \left( n, k \right) \right)} \geq C \left( b^2 - a^2 \right)^{\frac{\alpha}{2}} a^{\nu}.
			\end{equation}
			On the other hand, we have,
			\begin{equation}
				\label{AnnulusMass}
				\| f_{a, b} \|_{L^p_{\mu} \left( \mathbb{R}^n \right)} = C \left( b^{\mu + n} - a^{\mu + n} \right)^{\frac{1}{p}}.
			\end{equation}
			Using Equations \eqref{LInfinityEstimateAnnulus} and \eqref{AnnulusMass} in Inequality \eqref{RequiredInequality2}, we get for $0 < a < b$,
			$$a^{p \nu} \left( b^2 - a^2 \right)^{\frac{p \alpha}{2}} \leq C \left( b^{\mu + n} - a^{\mu + n} \right).$$
			Consider the function $H_a: \left( a, \infty \right) \rightarrow \left( 0, \infty \right)$, for a fixed $a > 0$, given by
			$$H_a \left( b \right) = \frac{\left( b^2 - a^2 \right)^{\frac{p \alpha}{2}}}{b^{\mu + n} - a^{\mu + n}}.$$
			Essentially, we require $H_a$ to be bounded. By taking $b \rightarrow a$ and using L'H\^{o}spital rule it is easy to see that the function is bounded only if $p \geq \frac{2}{\alpha}$.
		\end{example}
		We now remark that even in the boundary case $p = \frac{2}{\alpha}$, we cannot have Inequality \eqref{RequiredInequality2}. The argument is motivated from the works of Hardy and Littlewood in \cite{HardyLittlewoodFI}.
		\begin{remark}
			\label{FIRemark}
			\normalfont
			Let us assume that Inequality \eqref{RequiredInequality2} holds for all radial functions in $L^p_{\mu} \left( \mathbb{R}^n \right)$. This implies the following inequality of fractional integrals
			$$\| \rho_2 \left( \cdot \right) I^{\frac{\alpha}{2}}_-\varphi \|_{L^{\infty} \left( 0, \infty \right)} \leq C \| \varphi \|_{L^{\frac{2}{\alpha}} \left( \left( 0, \infty \right), \rho_1 \left( t \right) \mathrm{d}t \right)},$$
			where, $\rho_2 \left( t \right) = t^{\frac{\nu}{2}}$, and $\rho_1 \left( t \right) = t^{\frac{\mu + n}{2} - 1}$. We claim that such an inequality is impossible, and therefore Inequality \eqref{RequiredInequality2} is not possible.
			
			In fact, we show that $I^{\frac{\alpha}{2}}_-: L^{\frac{2}{\alpha}} \left( \left( 0, \infty \right), \rho_1 \left( t \right) \mathrm{d}t \right) \rightarrow L^{\infty} \left( \left( 0, \infty \right), \rho_2 \left( t \right) \mathrm{d}t \right)$ cannot be bounded for any choices of positive power functions $\rho_1$ and $\rho_2$. Here, we use the notation
			\begin{align*}
				L^{\infty} \left( \left( 0, \infty \right), \rho_2 \left( t \right) \mathrm{d}t \right) := \lbrace \psi: \left( 0, \infty \right) \rightarrow \mathbb{C} | &\psi \text{ is measurable and } \\
				&\| \rho_2 \left( \cdot \right) \psi \left( \cdot \right) \|_{L^{\infty} \left( 0, \infty \right)} < +\infty \rbrace.
			\end{align*}
			Throughout the argument, we shall use $\alpha < 2$. If, on the contrary, we have for all $\varphi \in L^{\frac{2}{\alpha}} \left( \left( 0, \infty \right), \rho_1 \left( t \right) \mathrm{d}t \right)$,
			$$\| \rho_2 \left( \cdot \right) I^{\frac{\alpha}{2}}_-\varphi \|_{L^{\infty} \left( 0, \infty \right)} \leq C \| \varphi \|_{L^{\frac{2}{\alpha}} \left( 0, \infty \right)},$$
			then, for any $\psi \in L^1 \left( \left( 0, \infty \right), \frac{\rho_1 \left( t \right)}{\rho_2 \left( t \right)} \mathrm{d}t \right)$, we have,
			\begin{align*}
				\int\limits_{0}^{\infty} I^{\frac{\alpha}{2}}_-\varphi \left( x \right) \psi \left( x \right) \rho_1 \left( x \right) \mathrm{d}x &= \int\limits_{0}^{\infty} \varphi \left( t \right) \int\limits_{0}^{t} \frac{\psi \left( x \right) \rho_1 \left( x \right)}{\left( t - x \right)^{1 - \frac{\alpha}{2}}} \mathrm{d}x \mathrm{d}t \\
				&= \int\limits_{0}^{\infty} \varphi \left( t \right) \frac{1}{\rho_1 \left( t \right)} I^{\frac{\alpha}{2}}_+ \left( \rho_1 \cdot \psi \right) \left( t \right) \rho_1 \left( t \right) \mathrm{d}t.
			\end{align*}
			Therefore, by duality, we have, $\frac{1}{\rho_1} I^{\frac{\alpha}{2}}_+ \left( \rho_1 \cdot \psi \right) \in L^q \left( \left( 0, \infty \right), \rho_1 \left( t \right) \mathrm{d}t \right)$, where $q = \left( \frac{2}{\alpha} \right)'$, whenever $\psi \in L^1 \left( \left( 0, \infty \right), \frac{\rho_1 \left( t \right)}{\rho_2 \left( t \right)} \mathrm{d}t \right)$. However, one can show by producing explicit examples that this is not true. We refer the reader to \cite{SelfLPImprovingKPlane} (Example C.3), where the authors have given an example for the case $\alpha = 1$. For an arbitrary $0 < \alpha < 2$, we consider the function $\psi$ of Example C.3 of \cite{SelfLPImprovingKPlane} with $1 < \gamma \leq 2 - \frac{\alpha}{2}$. Then, it is easy to see that $\psi \in L^1 \left( \left( 0, \infty \right), \rho_1 \right)$ but $I^{\frac{\alpha}{2}}_+ \psi \notin L^{\frac{2}{2 - \alpha}} \left( \left( 0, \infty \right), \rho_2 \right)$, for any choices of $\rho_1$ and $\rho_2$.
		\end{remark}
		Theorem \ref{SufficientLPLInfinity}, Example \ref{CounterExample3} and Remark \ref{FIRemark} completely answer Question \ref{Question2}.
	\section{End-Point Estimates for radial functions}
		\label{EndPointEstimateSection}
		We have seen in Theorem \ref{ExistenceTheorem} that the weighted $k$-plane transform $R_{\alpha}$ fails to exist on weighted Lebesgue space $L^{p}_{\mu} \left( \mathbb{R}^n \right)$, for $p \geq \frac{\mu + n}{\alpha}$. Therefore, the value $p = \frac{\mu + n}{\alpha}$ is an ``end-point" for the transform. In this section, we prove certain Lorentz space estimates for $R_{\alpha}$ at this end-point. These estimates tell us that although $R_{\alpha}$ is not defined on the Lebesgue space, it is well-defined on its subspace, namely $L^{p, 1}_{\mu} \left( \mathbb{R}^n \right)$. For a survey on Lorentz spaces $L^{p, q} \left( X \right)$, we refer the reader to \cite{GrfakosCFA}. 
		
		For the proof of our main result, we require the following lemma from \cite{KumarRayWtEst}.
		\begin{lemma}
			\label{EndPointLemmaKumarRay}
			For $\gamma > 0$ and characteristic functions $f$ of measurable sets of finite measure in $\left( 0, \infty \right)$, we have
			\begin{equation}
				\label{PolynomialInequalityLemma}
				\int\limits_{0}^{\infty} f \left( t \right) t^{\gamma - 1} \mathrm{d}t \leq p^{\frac{1}{p}} \gamma^{-\frac{1}{p'}} \left( \int\limits_{0}^{\infty} f \left( t \right) t^{p \gamma - 1} \mathrm{d}t \right)^{\frac{1}{p}}.
			\end{equation}
		\end{lemma}
		We now state and prove the main result of this section.
		\begin{theorem}
			\label{EndPointTheorem}
			Let $\alpha > 0$, $\mu > \alpha - n$ and $p = \frac{\mu + n}{\alpha}$. 
			\begin{enumerate}
				\item[\mylabel{AlphaMoreThan2EndPoint}{(A)}] For $\alpha \geq 2$ and all radial functions $f \in L^{p, 1}_{\mu} \left( \mathbb{R}^n \right)$, we have,
			\begin{equation}
				\label{EndPointInequality}
				\| \left( d \left( 0, \cdot \right) \right)^{\nu} R_{\alpha}f \|_{L^{\infty} \left( G \left( n, k \right) \right)} \leq C \| f \|_{L^{p, 1}_{\mu} \left( \mathbb{R}^n \right)},
			\end{equation}
			if and only if $\nu = 0$.
			\item[\mylabel{AlphaLessThan2EndPoint}{(B)}] For $0 < \alpha < 2$, Inequality \eqref{EndPointInequality} holds for all radial functions $f \in L^{p, 1}_{\mu} \left( \mathbb{R}^n \right)$ if and only if $\nu = 0$ and $p \geq \frac{2}{\alpha}$.
			\end{enumerate}
		\end{theorem}
		\begin{proof}
			First, we notice that since for characteristic functions of measurable sets, the $L^p$-norm and the Lorentz $L^{p, 1}$-norm coincide, the proof of necessity $\nu = 0$ is verbatim to that of Theorem \ref{NecessaryLPLInfinity}. We only prove the sufficiency of the conditions. To do so, it is enough to prove Inequality \eqref{EndPointInequality} for characteristic functions of radial measurable sets (see \cite{SteinWeissFA}). In this proof, therefore, we shall consider $f = \chi_A$, where $A \subseteq \mathbb{R}^n$ is measurable and radial.
			\begin{enumerate}
				\item[\mylabel{AlphaMoreThan2EndPointProof}{(A)}] We start with the case $\alpha \geq 2$, which is easy to handle. In this case, we have for $s > \| u \|$, $\left( 1 - \frac{\| u \|^2}{s^2} \right)^{\frac{\alpha}{2} - 1} \leq 1$. Therefore, we get,
				\begin{align*}
					\left| R_{\alpha}f \left( \pi, u \right) \right| & \leq C \int\limits_{\| u \|}^{\infty} \tilde{f} \left( t \right) t^{\alpha - 1} \mathrm{d}t.
				\end{align*}
				Using Lemma \ref{EndPointLemmaKumarRay}, we get
				\begin{align*}
					\left| R_{\alpha}f \left( \pi, u \right) \right| &\leq C \left( \int\limits_{0}^{\infty} \tilde{f} \left( t \right) t^{p \alpha - 1} \right)^{\frac{1}{p}} \\
					&= C \left( \int\limits_{0}^{\infty} \tilde{f} \left( t \right) t^{\mu + n - 1} \mathrm{d}t \right)^{\frac{1}{p}} = C \| f \|_{L^{p, 1}_{\mu} \left( \mathbb{R}^n \right)}.
				\end{align*}
				\item[\mylabel{AlphaLessThan2EndPointProof}{(B)}] When $0 < \alpha < 2$, the proof is a bit tricky since we can no longer dominate $\left( 1 - \frac{\| u \|^2}{s^2} \right)^{\frac{\alpha}{2} - 1}$ by a constant as in part \ref{AlphaMoreThan2EndPointProof}. To get around this, we substitute $s^2 - \| u \|^2 \left( 1 + x \right)$ in Equation \eqref{RadialFormulaEquation} and simplify to obtain
				\begin{align*}
					\left| R_{\alpha}f \left( \pi, u \right) \right| &= C \| u \|^{\alpha} \int\limits_{0}^{\infty} \tilde{f} \left( \| u \| \sqrt{1 + x} \right) x^{\frac{\alpha}{2} - 1} \mathrm{d}x.
				\end{align*}
				We notice that since $\tilde{f}$ is a characteristic function of a measurable set in $\left( 0, \infty \right)$, so is the function $\tilde{f} \left( \| u \| \sqrt{1 + \cdot} \right)$. Also, since $\alpha > 0$, Lemma \ref{EndPointLemmaKumarRay} is applicable. From Inequality \eqref{PolynomialInequalityLemma}, we get
				\begin{align*}
					\left| R_{\alpha}f \left( \pi, u \right) \right| &= C \| u \|^{\alpha} \left( \int\limits_{0}^{\infty} \tilde{f} \left( \| u \| \sqrt{1 + x} \right) x^{\frac{p \alpha}{2} - 1} \mathrm{d}x \right)^{\frac{1}{p}}.
				\end{align*}
				Resubstituting $t = \| u \| \sqrt{1 + x}$ and simplifying by using the fact that $p = \frac{\mu + n}{\alpha}$, we get
				\begin{align*}
					\left| R_{\alpha}f \left( \pi, u \right) \right| &\leq C \left( \int\limits_{\| u \|}^{\infty} \tilde{f} \left( t \right) \left( t^2 - \| u \|^2 \right)^{\frac{\mu + n}{2} - 1} t \ \mathrm{d}t \right)^{\frac{1}{p}}.
				\end{align*}
				Since $\mu + n = p \alpha \geq 2$, we have $\left( t^2 - \| u \|^2 \right)^{\frac{\mu + n}{2} - 1} \leq t^{\mu + n - 2}$. That is,
				\begin{align*}
					\left| R_{\alpha}f \left( \pi, u \right) \right| &\leq C \left( \int\limits_{0}^{\infty} \tilde{f} \left( t \right) t^{\mu + n - 1} \mathrm{d}t \right)^{\frac{1}{p}} = C \| f \|_{L^{p, 1}_{\mu} \left( \mathbb{R}^n \right)}.
				\end{align*}
			\end{enumerate}
		\end{proof}
	\section{Weighted \texorpdfstring{$L^p$}{Lp}-\texorpdfstring{$L^p$}{Lp} boundedness and the operator norm}
		\label{LPLPBoundednessSection}
		In this section, we discuss the weighted $L^p$-$L^p$ boundedness of the transforms $R_{\alpha}$. In this case, we also give the operator norm of the operators under consideration. The motivation and technique for the same lie in the work of Rubin (see \cite{RubinWtKPlane}). We begin by recalling a spherical representation of integrals over $\mathbb{R}^n$.
		\begin{lemma}[\cite{RubinWtKPlane}]
			\label{SphericalRepresentation}
			Let $f: \mathbb{R}^n \rightarrow \mathbb{C}$ be an integrable function. Then,
			\begin{equation}
				\label{SphericalRepresentaionEquatiion}
				\int\limits_{\mathbb{R}^n} f \left( x \right) \mathrm{d}x = \int\limits_{\mathbb{S}^n_+} f \left( \frac{\theta'}{\theta_{n + 1}} \right) \frac{\mathrm{d}\theta}{\theta_{n + 1}^{n + 1}},
			\end{equation}
			where, $\mathbb{S}^n_+ := \left\lbrace \omega \in \mathbb{S}^n | \omega_{n + 1} > 0 \right\rbrace$, and $\theta' = \left( \theta_1, \cdots, \theta_n \right)$.
		\end{lemma}
		Lemma \ref{SphericalRepresentation} leads us to a spherical representation of the transform $R_{\alpha}$.
		\begin{proposition}
			\label{SphericalRepresentationOurTransform}
			Let $\alpha > 0$, and $f: \mathbb{R}^n \rightarrow \mathbb{C}$ be a function such that $R_{\alpha}f$ is well-defined. Given $\left( \pi, u \right) \in G \left( n, k \right)$, let $g \in SO \left( n \right)$ be a rotation such that $g \cdot \mathbb{R}^k = \pi$ and $g \cdot e_{k + 1} = \frac{u}{\| u \|}$. Then,
			\begin{equation}
				\label{SphericalRepresentationRAlphaEquation}
				R_{\alpha}f \left( \pi, u \right) = \| u \|^{\alpha} \int\limits_{\mathbb{S}^k_+} f_g \left( \frac{\| u \| \theta}{\theta_{k + 1}} \right) \frac{\| \theta' \|^{\alpha - k}}{\theta_{k + 1}^{\alpha + 1}} \mathrm{d}\theta,
			\end{equation}
			where, $f_g \left( x \right) = f \left( g \cdot x \right)$.
		\end{proposition}
		\begin{proof}
			It is easy to see that
			$$R_{\alpha}f \left( \pi, u \right) = \int\limits_{\pi} f \left( x + u \right) \| x \|^{\alpha - k} \mathrm{d}x = \int\limits_{\mathbb{R}^k} f_g \left( \| u \| \left( \frac{x}{\| u \|} + e_{k + 1} \right) \right) \| x \|^{\alpha - k} \mathrm{d}x.$$
			Changing variables $x = \| u \| y$ and then using Equation \eqref{SphericalRepresentaionEquatiion}, we obtain Equation \eqref{SphericalRepresentationRAlphaEquation}.
		\end{proof}
		Using Equation \eqref{SphericalRepresentationRAlphaEquation} and an elementary technique of Minkowski's integral inequality, we now prove the weighted $L^p$-$L^p$ boundedness of the transform $R_{\alpha}$. Moving a step further, we also obtain the operator norm for $R_{\alpha}$.
		\begin{theorem}
			\label{LPLPBoundednessTheorem}
			Let $\alpha > 0$, $\mu > \alpha - n$ and $1 \leq p < \frac{\mu + n}{\alpha}$. Then, the operator $R_{\alpha}: L^p_{\mu} \left( \mathbb{R}^n \right) \rightarrow L^p_{\nu} \left( G \left( n, k \right) \right)$ is bounded if and only if $\nu = \mu + k - p \alpha$. Moreover, in this case, we have,
			\begin{equation}
				\label{OperatorNorm}
				\| R_{\alpha} \|_{L^p_{\mu} \left( \mathbb{R}^n \right) \rightarrow L^p_{\nu} \left( G \left( n, k \right) \right)} = \left( \frac{\left| \mathbb{S}^{n - k - 1} \right|}{\left| \mathbb{S}^{n - 1} \right|} \right)^{\frac{1}{p}} \frac{\left| \mathbb{S}^{k - 1} \right|}{2} \frac{\Gamma \left( \frac{\alpha}{2} \right) \Gamma \left( \frac{\mu + n}{2p} - \frac{\alpha}{2} \right)}{\Gamma \left( \frac{\mu + n}{2p} \right)}.
			\end{equation}
		\end{theorem}
		\begin{proof}
			We first remark that the necessity of $\nu = \mu + k - p \alpha$ follows from Equation \eqref{NecessaryConditionsLPLR} by taking $r = p$. Also, since $p < \frac{\mu + n}{\alpha}$, this condition automatically forces $\nu > k - n$. We shall first prove the boundedness of $R_{\alpha}$ and get an estimate on its operator norm.
			
			Using polar decomposition on $\pi^{\perp}$ (for a given $\pi \in G_{n, k}$), and recalling that we have considered the normalized measure on $G_{n, k}$, we obtain
			\begin{align*}
				&\| R_{\alpha}f \|_{L^p_{\nu} \left( G \left( n, k \right) \right)} \\
				&= \left| \mathbb{S}^{n - k - 1} \right|^{\frac{1}{p}} \left[ \int\limits_{SO \left( n \right)} \int\limits_{0}^{\infty} \left| r^{\alpha} \int\limits_{\mathbb{S}^k_+} f_g \left( \frac{r \theta}{\theta_{k + 1}} \right) \frac{\| \theta' \|^{\alpha - k}}{\theta_{k + 1}^{\alpha + 1}} \mathrm{d}\theta \right|^p r^{\nu + n - k - 1} \mathrm{d}r \mathrm{d}g \right]^{\frac{1}{p}}.
			\end{align*}
			Now, by applying Minkowski's integral inequality, we get
			\begin{align*}
				&\| R_{\alpha}f \|_{L^p_{\nu} \left( G \left( n, k \right) \right)} \\
				&\leq \left| \mathbb{S}^{n - k - 1} \right|^{\frac{1}{p}} \int\limits_{\mathbb{S}^k_+} \frac{\| \theta' \|^{\alpha - k}}{\theta_{k + 1}^{\alpha + 1}} \left[ \int\limits_{SO \left( n \right)} \int\limits_{0}^{\infty} \left| f_g \left( \frac{r \theta}{\theta_{k + 1}} \right) \right|^p r^{\nu + n - k - 1 + p \alpha} \mathrm{d}r \mathrm{d}g \right]^{\frac{1}{p}} \mathrm{d}\theta.
			\end{align*}
			Next, we substitute $r = \theta_{k + 1} t$ in the innermost integral, use the fact that $\nu = \mu + k - p \alpha$, and simplify to obtain
			\begin{align}
				&\| R_{\alpha}f \|_{L^p_{\nu} \left( G \left( n, k \right) \right)} \nonumber \\
				&\leq \left| \mathbb{S}^{n - k - 1} \right|^{\frac{1}{p}} \int\limits_{\mathbb{S}^k_+} \| \theta' \|^{\alpha - k} \theta_{k + 1}^{\frac{\nu + n - k}{p} - 1} \left[ \int\limits_{SO \left( n \right)} \int\limits_{0}^{\infty} \left| f_g \left( t \theta \right) \right|^p t^{\mu + n - 1} \mathrm{d}t \mathrm{d}g \right]^{\frac{1}{p}} \mathrm{d}\theta \nonumber \\
				\label{PenultipmateLPLPBounded}
				&= \left( \frac{\left| \mathbb{S}^{n - k - 1} \right|}{\left| \mathbb{S}^{n - 1} \right|} \right)^{\frac{1}{p}} \| f \|_{L^p_{\mu} \left( \mathbb{R}^n \right)} \int\limits_{\mathbb{S}^k_+} \| \theta' \|^{\alpha - k} \theta_{k + 1}^{\frac{\mu + n}{p} - \alpha - 1} \mathrm{d}\theta.
			\end{align}
			Since $p < \frac{\mu + n}{\alpha}$, it is clear that the above integral on the hemisphere ($\mathbb{S}^k_+$) is finite. Moreover, from the polar decomposition of the sphere, we easily see that
			\begin{equation}
				\label{BetaIntegralHemisphere}
				\int\limits_{\mathbb{S}^k_+} \| \theta' \|^{\alpha - k} \theta_{k + 1}^{\frac{\mu + n}{p} - \alpha - 1} \mathrm{d}\theta = \frac{\left| \mathbb{S}^{k - 1} \right|}{2} \frac{\Gamma \left( \frac{\alpha}{2} \right) \Gamma \left( \frac{\mu + n}{2p} - \frac{\alpha}{2} \right)}{\Gamma \left( \frac{\mu + n}{2p} \right)}.
			\end{equation}
			Equations \eqref{PenultipmateLPLPBounded} and \eqref{BetaIntegralHemisphere}  prove the weighted $L^p$-$L^p$ boundedness of $R_{\alpha}$, and give
			$$\| R_{\alpha} \|_{L^p_{\mu} \left( \mathbb{R}^n \right) \rightarrow L^p_{\nu} \left( \mathbb{R}^n \right)} \leq \left( \frac{\left| \mathbb{S}^{n - k - 1} \right|}{\left| \mathbb{S}^{n - 1} \right|} \right)^{\frac{1}{p}} \frac{\left| \mathbb{S}^{k - 1} \right|}{2} \frac{\Gamma \left( \frac{\alpha}{2} \right) \Gamma \left( \frac{\mu + n}{2p} - \frac{\alpha}{2} \right)}{\Gamma \left( \frac{\mu + n}{2p} \right)}.$$
			
			We now prove Equation \eqref{OperatorNorm}. To do so, we first observe (by using H\"{o}lder's inequality) that
			\begin{equation}
				\label{DualEstimate}
				\left| \int\limits_{G \left( n, k \right)} R_{\alpha}f \left( \xi \right) \varphi \left( \xi \right) \mathrm{d}\xi \right| \leq \| R_{\alpha} \|_{L^p_{\mu} \left( \mathbb{R}^n \right) \rightarrow L^p_{\nu} \left( G \left( n, k \right) \right)} \| f \|_{L^p_{\mu} \left( \mathbb{R}^n \right)} \| \varphi \|_{L^{p'}_{- \frac{p' \nu}{p}} \left( G \left( n, k \right) \right)},
			\end{equation}
			for all nice functions $f: \mathbb{R}^n \rightarrow \mathbb{C}$ and $\varphi: G \left( n, k \right) \rightarrow \mathbb{C}$. When $p = 1$, the norm $\| \varphi \|_{L^{p'}_{- \frac{p' \nu}{p}} \left( G \left( n, k \right) \right)}$ is to be interpreted as $\| \left( d \left( 0, \cdot \right) \right)^{-\nu} \varphi \|_{L^{\infty} \left( G \left( n, k \right) \right)}$.
			
			Now, we consider the function $f: \mathbb{R}^n \rightarrow \mathbb{C}$, defined by
			$$f \left( x \right) = \| x \|^{- \frac{\mu + n}{p} - \delta} \chi_{\mathbb{R}^n \setminus B \left( 0, 1 \right)} \left( \| x \| \right).$$
			Here, $\delta > 0$ is fixed arbitrarily. Then, we have, by using polar decomposition on $\mathbb{R}^n$,
			$$\| f \|_{L^p_{\mu} \left( \mathbb{R}^n \right)}^p = \left| \mathbb{S}^{n - 1} \right| \int\limits_{1}^{\infty} t^{-p \delta - 1} \mathrm{d}t = \frac{\left| \mathbb{S}^{n - 1} \right|}{p \delta}.$$
			We observe that $f$ is a radial function. We denote $\tilde{f}: \left( 0, \infty \right) \rightarrow \mathbb{C}$, the function so that $f \left( x \right) = \tilde{f} \left( \| x \| \right)$. That is, $\tilde{f} \left( t \right) = t^{- \frac{\mu + n}{p} - \delta} \chi_{\left( 1, \infty \right)} \left( t \right)$.
			
			Next, let us consider $\varphi: G \left( n, k \right) \rightarrow \mathbb{C}$, defined by
			$$\varphi \left( \xi \right) = \left( d \left( 0, \xi \right) \right)^{\mu + k - \alpha} \left( \tilde{f} \left( d \left( 0, \xi \right) \right) \right)^{p - 1}.$$
			Then, we see from the polar decomposition on $\pi^{\perp}$, for a given $\pi \in G_{n, k}$, that
			\begin{align*}
				\| \varphi \|_{L^{p'}_{- \frac{p' \nu}{p}} \left( G \left( n, k \right) \right)} &= \left| \mathbb{S}^{n - k - 1} \right|^{\frac{1}{p'}} \left( \int\limits_{0}^{\infty} \left| \tilde{f} \left( t \right) \right|^{p' \left( p - 1 \right)} t^{\mu + n - 1} \mathrm{d}t \right)^{\frac{1}{p'}} \\
				&= \frac{\left| \mathbb{S}^{n - k - 1} \right|^{\frac{1}{p'}}}{\left| \mathbb{S}^{n - 1} \right|^{\frac{1}{p'}}} \| f \|_{L^p_{\mu} \left( \mathbb{R}^n \right)}^{p - 1}.
			\end{align*}
			
			Now, since both $f$ and $\varphi$ are radial, we have from Equation \eqref{SphericalRepresentationRAlphaEquation}
			\begin{align*}
				&\int\limits_{G \left( n, k \right)} R_{\alpha}f \left( \xi \right) \varphi \left( \xi \right) \mathrm{d}\xi \\
				&= \left| \mathbb{S}^{n - k - 1} \right| \int\limits_{0}^{\infty} r^{\alpha + n - k - 1} \tilde{\varphi} \left( r \right) \int\limits_{\mathbb{S}^k_+} \tilde{f} \left( \frac{r}{\theta_{k + 1}} \right) \frac{\| \theta' \|^{\alpha - k}}{\theta_{k + 1}^{\alpha + 1}} \mathrm{d}\theta \mathrm{d}r \\
				&= \left| \mathbb{S}^{n - k - 1} \right| \int\limits_{\mathbb{S}^k_+} \frac{\| \theta' \|^{\alpha - k}}{\theta_{k + 1}^{\alpha + 1}} \left( \int\limits_{0}^{\infty} \tilde{f} \left( \frac{r}{\theta_{k + 1}} \right) \tilde{\varphi} \left( r \right) r^{\alpha + n - k - 1} \mathrm{d}r \right) \mathrm{d}\theta.
			\end{align*}
			Using the definitions of $f$ and $\varphi$, we evaluate the inner integral to be
			$$\int\limits_{0}^{\infty} \tilde{f} \left( \frac{r}{\theta_{k + 1}} \right) \tilde{\varphi} \left( r \right) r^{\alpha + n - k - 1} \mathrm{d}r = \frac{\theta_{k + 1}^{\frac{\mu + n}{p} + \delta}}{p \delta}.$$
			Using these observations in Inequality \eqref{DualEstimate}, we get
			\begin{align*}
				&\left| \mathbb{S}^{n - k - 1} \right| \int\limits_{\mathbb{S}^k_+} \frac{\| \theta' \|^{\alpha - k}}{\theta_{k + 1}^{\alpha + 1}} \frac{\theta_{k + 1}^{\frac{\mu + n}{p} + \delta}}{p \delta} \\
				&\leq \left( \frac{\left| \mathbb{S}^{n - k - 1} \right|}{\left| \mathbb{S}^{n - 1} \right|} \right)^{\frac{1}{p'}} \frac{\left| \mathbb{S}^{n - 1} \right|}{p \delta} \| R_{\alpha} \|_{L^p_{\mu} \left( \mathbb{R}^n \right) \rightarrow L^p_{\nu} \left( G \left( n, k \right) \right)}.
			\end{align*}
			Upon simplifying, we get
			\begin{align*}
				\| R_{\alpha} \|_{L^p_{\mu} \left( \mathbb{R}^n \right) \rightarrow L^p_{\nu} \left( G \left( n, k \right) \right)} &\geq \left( \frac{\left| \mathbb{S}^{n - k - 1} \right|}{\left| \mathbb{S}^{n - 1} \right|} \right)^{\frac{1}{p}} \int\limits_{\mathbb{S}^k_+} \| \theta' \|^{\alpha - k} \theta_{k + 1}^{\frac{\mu + n}{p} + \delta - \alpha - 1} \mathrm{d}\theta \\
				&= \left( \frac{\left| \mathbb{S}^{n - k - 1} \right|}{\left| \mathbb{S}^{n - 1} \right|} \right)^{\frac{1}{p}} \frac{\left| \mathbb{S}^{k - 1} \right|}{2} \frac{\Gamma \left( \frac{\alpha}{2} \right) \Gamma \left( \frac{\mu + n}{2p} - \frac{\alpha}{2} + \frac{\delta}{2} \right)}{\Gamma \left( \frac{\mu + n}{2p} + \frac{\delta}{2} \right)},
			\end{align*}
			for every $\delta > 0$. By taking $\delta \rightarrow 0$ and using the continuity of Gamma function, we obtain
			\begin{align*}
				\| R_{\alpha} \|_{L^p_{\mu} \left( \mathbb{R}^n \right) \rightarrow L^p_{\nu} \left( G \left( n, k \right) \right)} &\geq \left( \frac{\left| \mathbb{S}^{n - k - 1} \right|}{\left| \mathbb{S}^{n - 1} \right|} \right)^{\frac{1}{p}} \frac{\left| \mathbb{S}^{k - 1} \right|}{2} \frac{\Gamma \left( \frac{\alpha}{2} \right) \Gamma \left( \frac{\mu + n}{2p} - \frac{\alpha}{2} \right)}{\Gamma \left( \frac{\mu + n}{2p} \right)}.
			\end{align*}
			This completes the proof!
		\end{proof}
	\section{Conclusion}
		\label{ConclusionSection}
		In this article we have introduced and studied a weighted analogue of the $k$-plane transform, where the weight is a radial power weight. Particularly, we have obtained necessary and sufficient conditions for the said transform to be bounded on weighted Lebesgue spaces of radial functions. As discussed in Section \ref{IntroductionSection}, the operator defined in Equation \eqref{OurTransformDefinition} has certain interesting interpretations. As a next step, one could consider the problem of inversion of the transforms $R_{\alpha}$, using various techniques such as the convolution-backprojection, wavelet and Fourier transforms, and fractional integrals.
		
		In the direction of studying the mapping properties of $R_{\alpha}$, the problem of weighted $L^p$-$L^q$ boundedness for general (not necessarily radial) functions remains open. Once such a boundedness is obtained, one may try and obtain the operator norm of $R_{\alpha}$.
		
		The definition of the operator $R_{\alpha}$, given in Equation \eqref{OurTransformDefinition}, also opens up another direction of study. To motivate it, let us first observe that the weight $\| x \|^{\alpha - k}$ that is attached to the measure of $k$-planes has a (blow-up) singularity when $\alpha < k$. Also, we cannot expect the integral of Equation \eqref{OurTransformDefinition} to exist when $\alpha = 0$. Indeed, when $\alpha = 0$, the definition looks much like a singular integral over the $k$-planes, and we need to introduce a kernel that has nice cancellation properties. Particularly, we may define the operator
		$$T_{\alpha}f \left( \pi, u \right) = \int\limits_{\mathbb{R}^k} f_g \left( \| u \|e_{k + 1} - x \right) \frac{\Omega \left( \frac{x}{\| x \|} \right)}{\| x \|^k} \mathrm{d}x,$$
		where, $g \in SO \left( n \right)$ is such that $g \cdot \mathbb{R}^k = \pi$ and $g \cdot e_{k + 1} = \frac{u}{\| u \|}$, and $\Omega: \mathbb{S}^{k - 1} \rightarrow \mathbb{C}$. is a nice function that satisfies certain ``cancellation properties". Such operators can be interpreted as singular integrals associated to $k$-planes and are not much studied in the available literature, hence deserve further investigation.




	\section*{Acknowledgements}
		The first author extends his gratitude towards the Prime Minister Research Fellowship (PMRF), grant number: 2101706, funded by the Ministry of Education (MoE), Government of India.

\end{document}